\documentclass[pdflatex]{sn-jnl}

\usepackage{defaults}
\usepackage{amsmath}
\usepackage{amssymb}
\usepackage{amsthm}
\usepackage{mathtools}
\usepackage{bbm}
\usepackage{xcolor}
\usepackage{hyperref}
\usepackage{subcaption}
\usepackage{cleveref}

\newcommand{\CC}{\mathbb{C}}
\newcommand{\CP}{\mathbb{C}\mathbb{P}}

\newcommand{\rmep}[1]{\mcl{A}(#1)}
\newcommand{\irmep}[1]{\mcl{B}(#1)}
\newcommand{\smep}[2][i]{\mcl{V}_{#1}(#2)}

\newcommand{\sol}[2][]{#2^{\circ #1}}
\newcommand{\eig}{\vc{\lambda}}
\newcommand{\bmu}{\vc{\mu}}
\newcommand{\bomega}{\vc{\omega}}
\newcommand{\heig}{\vctilde{\lambda}}
\newcommand{\rvec}{\vc{x}}
\newcommand{\lvec}{\vc{y}}
\newcommand{\rmat}{\mt{A}}
\newcommand{\irmat}{\mt{B}}

\newcommand{\homo}[1]{H(#1, \heig, \rvec)}
\newcommand{\elem}[2]{b_{#2}^{(#1)}}

\renewcommand{\arg}{\mathop{\rm arg}\nolimits}
\renewcommand{\herm}{*}

\theoremstyle{plain} 
\newtheorem{theorem}{Theorem}[section]
\newtheorem{corollary}[theorem]{Corollary}

\newtheorem{lemma}[theorem]{Lemma}
\theoremstyle{definition} 
\newtheorem{definition}[theorem]{Definition}
\newtheorem{example}[theorem]{Example}
\theoremstyle{remark} 

\newtheorem{remark}[theorem]{Remark}

\makeatletter
\DeclareRobustCommand{\doiurl}[1]{}
\makeatother

\begin{document}
    \title[Homotopy for Rectangular Multiparameter Eigenvalue Problems]{Homotopy for Rectangular Multiparameter Eigenvalue Problems}

    \author*[1,2]{\fnm{Bor} \sur{Plestenjak}}\email{bor.plestenjak@fmf.uni-lj.si}

    \author[3]{\fnm{Christof} \sur{Vermeersch}}\email{christof.vermeersch@esat.kuleuven.be}

    \affil*[1]{\orgdiv{Faculty of Mathematics and Physics}, \orgname{University of Ljubljana}, \orgaddress{\street{Jadranska 19}, \city{1000 Ljubljana}, \country{Slovenia}}}
    \affil*[2]{\orgname{Institute of Mathematics, Physics and Mechanics}, \orgaddress{\street{Jadranska 19}, \city{1000 Ljubljana}, \country{Slovenia}}}

    \affil[3]{\orgdiv{Center for Dynamical Systems, Signal Processing, and Data Analytics (STADIUS), Department of Electrical Engineering (ESAT)}, \orgname{KU Leuven}, \orgaddress{\street{Kasteelpark Arenberg 10}, \city{Leuven}, \postcode{3001}, \country{Belgium}}}

    \abstract{
        The rectangular multiparameter eigenvalue problem (rMEP) is a recently introduced form of an
        eigenvalue problem, in which we are looking for the combinations of parameters where  the multivariate rectangular matrix
        loses full column rank.
        Existing numerical methods for rMEPs typically require substantial memory and are not suitable for problems with a large number of eigenvalues.
        We propose a novel numerical approach based on the homotopy method that has low memory requirements and
        can compute all eigenvalues of an rMEP.
        Moreover, the homotopy method is naturally parallelizable, providing a simple and efficient way to accelerate computations.
        An implementation, together with a collection of test examples, is available in a public repository.
    }



\pacs[Funding]{The work of the first author was partially supported by the Slovenian Research and Innovation Agency (grant P1-0294).}

    \maketitle

    \section{Introduction}
        \label{sec:introduction}

        In contrast to standard multiparameter eigenvalue problems with square coefficient matrices, interest in solving rectangular problems is relatively recent and
        has been motivated by problems in system identification~\cite{vermeersch2019globally} and model order reduction~\cite{lagauw2023globally}. For these applications, the coefficient matrices can become very large, creating a need for efficient numerical methods
        with modest memory requirements.

        Homotopy methods provide such an approach. They have been successfully used to solve one-parameter eigenvalue problems~\cite{li1993homotopy, li1987homotopy} and, more recently, standard square multiparameter eigenvalue problems~\cite{DongYuYu_SIMAX, P_ContRD2EP,  Rodriguez_Fiber_Homotopy}.
        In this paper, we extend these methods to rectangular multiparameter eigenvalue problems. The extension is not straightforward, with the main challenge being the construction of a suitable initial problem whose eigenpairs are known and from which all solution paths can be tracked.

        Before outlining the paper and highlighting its contributions (\cref{sec:outline}), we recall the definitions of the rectangular (\cref{sec:rmep}) and square (\cref{sec:smep}) multiparameter eigenvalue problems.

        \subsection{Rectangular multiparameter eigenvalue problems}
            \label{sec:rmep}

            \begin{definition}
                A \emph{rectangular multiparameter eigenvalue problem (rMEP)} with $k \ge 2$
                parameters has the form
                \begin{equation}
                    \label{rectmep}
                    \rmep{\eig} \rvec \coloneq \sum_{\left\lbrace\bomega\right\rbrace} {\eig}^{\bomega} \rmat_{\vc{\omega}} \rvec = \vc{0},
                \end{equation}
                where $\vc{\omega} = (\omega_1, \ldots, \omega_k)$ is a multi-index and $\rmat_{\vc{\omega}} = \rmat_{\omega_1, \ldots, \omega_k}$ is an $(n + k - 1) \times n$ coefficient matrix over $\CC$.
                A tuple $\sol{\eig} = (\sol{\lambda_1}, \ldots, \sol{\lambda_k}) \in \CC^k$ is an \emph{eigenvalue} of~\eqref{rectmep} if $\rmep{\sol{\eig}} \sol{\rvec} = \vc{0}$ for a nonzero $\sol{\rvec} \in \CC^n$; such $\sol{\rvec}$ is the corresponding (right) \emph{eigenvector}.
            \end{definition}

            We assume that $\mcl{A}$ is regular, which means that the normal rank of $\mcl{A}$, defined as
            $
                \nrank(\mcl{A}) \coloneq \max_{\eig\in \CC^k} \, \rank(\rmep{\eig}),
            $
            is full, i.e., $\nrank(\mcl{A}) = n$.
            Then $\sol{\eig}\in \CC^k$ is an eigenvalue of~\eqref{rectmep} if and only if $\rank(\rmep{\sol{\eig}}) < n$.

            The degree $d$ of~\eqref{rectmep} is the maximal total degree of all monomials ${\eig}^{\bomega}$ in~\eqref{rectmep}, i.e., $d = \max_{\vc{\omega}} (\omega_1 + \cdots + \omega_k)$.
            We know that a generic rMEP of degree $d$ has
            \begin{equation}
                \label{number_of_solutions}
                d^k \binom{n + k - 1}{k}
            \end{equation}
            distinct eigenvalues~\cite[Theorem 3]{HKP_NLAA_2023_RectMEP}.
            If $d = 1$, then~\eqref{rectmep} is a \emph{linear rMEP}, for which the previous bound on the number of solutions was proven in~\cite{ShapiroShapiro}.

            \begin{definition}
                \label{linrMEPB}
                A linear rMEP
                $
                    \rmep{\eig} \rvec \coloneq \left(\rmat_0 + \lambda_1 \rmat_1 + \cdots + \lambda_k \rmat_k\right) \rvec = \vc{0},
                $
                where $\rmat_i \in \CC^{(n + k - 1) \times n}$ for $i = 0, \ldots, k$,
                is
                \emph{finitely regular}  if $\rank(\lambda_1 \rmat_1 + \cdots + \lambda_k \rmat_k)=n$ for all $\eig \ne \vc{0}$.
            \end{definition}

            It follows from~\cite[Lemma 1]{ShapiroShapiro} that a finitely regular linear rMEP has exactly $\binom{n+k-1}{k}$ finite eigenvalues counting multiplicities, which generalizes a well-known property that an $n \times n$ generalized eigenvalue problem (GEP) $\rmat \rvec=\lambda \irmat \rvec$ with nonsingular $\irmat$ has exactly $n$ finite eigenvalues.

            Existing numerical methods for rMEPs
            rely on structured matrices.
            The first approach is based on block Macaulay matrices~\cite{VermeerschDeMoorLAA}, while the second one uses operator determinants with the Kronecker product~\cite{HKP_NLAA_2023_RectMEP}.
            Although they seem to be unrelated, both approaches solve the rMEP~\eqref{rectmep} by constructing an associated family of GEPs
            \begin{equation}
                \label{deltamep}
                \mt{G}_i \vc{z} = \lambda_i \mt{G}_0 \vc{z}, \quad i = 1, \ldots, k,
            \end{equation}
            such that $\mt{G}_0$ is nonsingular, the matrices $\mt{G}_0^\inv \mt{G}_1, \ldots, \mt{G}_0^\inv \mt{G}_k$ commute, and $\sol{\eig} = (\sol{\lambda_1}, \ldots, \sol{\lambda_k})$ is an eigenvalue of~\eqref{rectmep} if and only if
            $(\sol{\lambda_1}, \ldots, \sol{\lambda_k})$ is a joint eigenvalue of~\eqref{deltamep}.
            Since the size of the matrices $\mt{G}_i$ in~\eqref{deltamep} is given by~\eqref{number_of_solutions}, and intermediate matrices used in the construction can be even larger, these methods can be applied in practice only to problems for which the number of solutions~\eqref{number_of_solutions} is not too large.

            Suppose that the rMEP is so large that, due to memory constraints, it is not possible to explicitly compute the coefficient matrices of the associated family of GEPs~\eqref{deltamep}.
            To compute all eigenvalues, a different approach with smaller memory requirements
            is needed.
            Tackling that issue is the goal of the novel homotopy method in the paper.

        \subsection{Standard square multiparameter eigenvalue problems}
            \label{sec:smep}

            \begin{definition}
                A standard \emph{square multiparameter eigenvalue problem (sMEP)} with $k \ge 2$ 
                parameters is given by $k$ equations of the form
                \begin{equation}
                    \label{problem}
                    \smep{\eig} \rvec_i \coloneq \left(\mt{V}_{i0} + \lambda_1 \vc{V}_{i1} + \cdots + \lambda_k \vc{V}_{ik}\right) \rvec_i = \vc{0}, \quad i = 1, \ldots, k,
                \end{equation}
                where $\mt{V}_{ij}\in \CC^{n_i\times n_i}$ is a square matrix.
                If $\smep{\sol{\eig}} \sol{\rvec_i} = \vc{0}$ for $i = 1, \ldots, k$ and nonzero $\sol{\rvec_i} \in \CC^{n_i}$,
                then
                $\sol{\eig} = (\sol{\lambda_1}, \ldots, \sol{\lambda_k}) \in \CC^k$ is an \emph{eigenvalue} of~\eqref{problem} and
                $\sol{\rvec} = \sol{\rvec_1} \otimes \cdots \otimes \sol{\rvec_k}$, where $\otimes$ denotes the Kronecker product,
                is the corresponding
                \emph{eigenvector}.
            \end{definition}

            A generic sMEP~\eqref{problem} has $n_1 \cdots n_k$ eigenvalues that are the common roots of the system of $k$ multivariate characteristic polynomial equations $\det(\smep{\eig}) = 0$ for $i = 1, \ldots, k$.
            The sMEP~\eqref{problem} is related to an associated system of GEPs,
            \begin{equation}
                \label{drugi}
                \Delta_i \rvec = \lambda_i \Delta_0 \rvec, \quad i = 1, \ldots, k,
            \end{equation}
            where the $(n_1 \cdots n_k) \times (n_1 \cdots n_k)$ matrices
            \[
                \Delta_0 =
                \left|\begin{matrix}
                    \mt{V}_{11} & \cdots & \mt{V}_{1k} \\[-0.5mm]
                    \vdots & & \vdots \\[-0.5mm]
                    \mt{V}_{k1} & \cdots & \mt{V}_{kk}
                \end{matrix}\right|_\otimes \coloneq  \sum_{\sigma\in S_k}{\rm sgn}(\sigma) \, \mt{V}_{1\sigma_1}\otimes \mt{V}_{2\sigma_2}\otimes \cdots \otimes \mt{V}_{k\sigma_k},
            \]
            and
            \[
                \Delta_i = -\
                \left|\begin{matrix}
                    \mt{V}_{11} & \cdots & \mt{V}_{1,i-1} & \mt{V}_{10} & \mt{V}_{1,i+1} & \cdots & \mt{V}_{1k} \\[-0.5mm]
                    \vdots & & \vdots & \vdots & \vdots & & \vdots \\[-0.5mm]
                    \mt{V}_{k1} & \cdots & \mt{V}_{k,i-1} & \mt{V}_{k0} & \mt{V}_{k,i+1} & \cdots & \mt{V}_{kk}
                \end{matrix}\right|_\otimes,\quad i=1,\ldots,k,
            \]
            are called \emph{operator determinants}, for more details, see, e.g.,~\cite{AtkinsonBook}.
            If $n_1 \cdots n_k$ is too large, then again, due to memory limitations, it is not possible to explicitly construct
            the operator determinants $\Delta_i$ of
            the associated GEPs~\eqref{drugi}.
            A numerical method with low memory requirements that can compute all eigenvalues of such large sMEPs is the homotopy method, see, e.g.,~\cite{DongYuYu_SIMAX} and~\cite{Rodriguez_Fiber_Homotopy}.

        \subsection{Outline and contributions}
            \label{sec:outline}

            Before discussing the novel homotopy method, we give some new theoretical results on
             eigenvalues of rMEPs in \cref{sec:auxiliaryresults}.
            Afterwards, we consider the homotopy method for linear and polynomial rMEPs in \cref{sec:linear} and \cref{sec:polynomial}, respectively.
            In \cref{sec:numericalexamples}, we test our numerical method on different rMEPs with a zero-dimensional and a positive-dimensional solution set, and show that the homotopy method
            can be applied to
            large rMEPs.
            Finally, we conclude the paper and outline directions for future work in \cref{sec:conclusion}.

            To the best of our knowledge, the paper presents the first application of a homotopy method to
            an rMEP.
            In addition, we demonstrate that the homotopy method is more efficient, sometimes even by a large margin, than the methods from~\cite{VermeerschDeMoorLAA} and~\cite{HKP_NLAA_2023_RectMEP}, which are implemented in~\cite{macaulaylab} and~\cite{multipareig_29}, respectively.

    \section{Auxiliary results}
        \label{sec:auxiliaryresults}

        In this section, we introduce some results on eigenvalues and eigenvectors of rMEPs that we need later when constructing the homotopy.
        The results in \cref{thm:polysimpleeig} and Corollary \ref{thm:simpleeig} are, to our knowledge, new and have not been published before.
        They complement similar results for sMEPs and give additional insight into the properties of rMEPs, for which the theory is much less developed  than for sMEPs.

        \subsection{Left null space and left eigenvectors}
            \label{sec:lefteigenvectors}

            Let $\sol{\eig}\in\CC^k$ be an eigenvalue of a regular rMEP~\eqref{rectmep}.
            Its geometric multiplicity is
            $\dim(\ker(\rmep{\sol{\eig}}))$ and we say that $\sol{\eig}$ is geometrically simple if $\dim(\ker(\rmep{\sol{\eig}})) = 1$.
            Since the $(n + k - 1 ) \times n$ coefficient matrices of an rMEP are rectangular and tall, there exist
            nonzero polynomial vectors $\lvec(\eig)$ such that
            $ \lvec(\eig)^\tr \rmep{\eig} \equiv 0$, see \cite[Lemma 3.1]{Khazanov_SpectralPropMEP_1995}.
            The set of all these vectors forms the left null space
            \[
                {\cal N}_l(\mcl{A}) \coloneq \big\{\lvec(\eig) \in \CC[\eig]^{n + k - 1}: \lvec(\eig)^\tr \rmep{\eig} \equiv \vc{0}\big\}.
            \]
            %
            The left singular space of $\mcl{A}$ at $\bmu \in \CC^k$, which is a subset
            of $\ker(\rmep{\bmu}^\herm)$,
            is then
            \[
                \ker_{\bmu}(\mcl{A}) \coloneq \big\{\overline{\lvec(\bmu)}: \lvec(\eig)\in{\cal N}_l(\mcl{A})\big\}
            \]
            and $\dim(\ker_{\bmu}(\mcl{A}))\le k-1$.
            If $\sol{\eig}$ is an eigenvalue, then $\dim(\ker(\rmep{\sol{\eig}}^\herm)) \ge k$ and each vector $\lvec \not\in \ker_{\sol{\eig}}(\mcl{A})$ such that $\lvec^\herm \rmep{\sol{\eig}} = \vc{0}$ represents a left eigenvector.
            If $\sol{\eig}$ is geometrically simple, then
            $k$ linearly independent vectors $\vc{z}_1, \ldots, \vc{z}_k$ exist such that $\vc{z}_i^\herm \rmep{\sol{\eig}} = \vc{0}$ for $i = 1, \ldots, k$.

        \subsection{Nonsingularity of the Jacobian matrix}
            \label{sec:jacobian}

            Let $f_{j_1, \ldots, j_n}(\eig) \coloneq \det(\mcl{A}_{j_1, \ldots, j_n}(\eig))$ be an $n \times n$ minor of $\rmep{\eig}$ from the rows $j_1, \ldots, j_n$, where $1 \le j_1 < \ldots <j_n \le n + k - 1$.
            Notice that there are $\binom{n + k - 1}{n}$ minors of this kind.
            Then $\sol{\eig} \in \CC^k$ is an eigenvalue of $\rmep{\eig}$, see, e.g.,~\cite{Khazanov_SpectralPropMEP_1995}, if and only if
            \begin{equation}
                \label{eq:eig_system_minors}
                f_{j_1, \ldots, j_n}(\sol{\eig}) = 0 \quad \textrm{for all }\ 1 \le j_1 < \ldots <j_n \le n + k - 1.
            \end{equation}
            The eigenvalue $\sol{\eig}$ is algebraically simple
            if and only if the Jacobian of~\eqref{eq:eig_system_minors} has full rank at $\sol{\eig}$, where the Jacobian is an $\binom{n+k-1}{n} \times k$ matrix with rows
            \begin{equation}
                \label{part_der_minors}
                \left[\begin{matrix}
                    \frac{\partial f_{j_1,\ldots,j_n}}{\partial \lambda_1}(\eig) & \cdots & \frac{\partial f_{j_1,\ldots,j_n}}{\partial \lambda_k}(\eig)
                \end{matrix}\right],\quad 1\le j_1<\ldots<j_n\le n+k-1.
            \end{equation}
            Each algebraically simple eigenvalue is also geometrically simple, while the opposite is not always true~\cite{Khazanov_SpectralPropMEP_1995}.
            The following theorem gives a simple condition that a geometrically simple eigenvalue must satisfy to also be algebraically simple.

            \begin{theorem}
                \label{thm:polysimpleeig}
                Let $\sol{\eig} \in \CC^k$ be an algebraically simple eigenvalue of a regular rMEP $\rmep{\eig}$.
                If $\sol{\rvec} \in \CC^n$ is the corresponding right eigenvector, then the matrix
                \begin{equation}
                    \label{genB0_condition}
                    \left[\begin{matrix}
                        \vc{z}_1^\herm \frac{\partial {\bm{\mathcal A}}}{\partial \lambda_1} ({\sol{\eig}})\sol{\rvec} & \cdots & \vc{z}_1^\herm \frac{\partial {\bm{\mathcal A}}}{\partial \lambda_k} ({\sol{\eig}})\sol{\rvec} \\
                        \vdots & & \vdots \\
                        \vc{z}_k^\herm \frac{\partial {\bm{\mathcal A}}}{\partial \lambda_1} ({\sol{\eig}})\sol{\rvec} & \cdots & \vc{z}_k^\herm \frac{\partial {\bm{\mathcal A}}}{\partial \lambda_k} ({\sol{\eig}})\sol{\rvec}
                    \end{matrix}\right]
                \end{equation}
                is nonsingular for any $k$ linearly independent vectors $\vc{z}_1, \ldots, \vc{z}_k$ from $\ker(\rmep{\sol{\eig}}^\herm)$.
            \end{theorem}

            \begin{proof}
                Let $\mt{P}_{j_1, \ldots, j_n}=\left[\begin{matrix}e_{j_1} & e_{j_2} & \cdots & e_{j_n}\end{matrix}\right]$, where $1 \le j_1 < \ldots < j_n \le n+k-1$, be an $(n+k-1)\times n$ matrix
                whose columns are standard basis vectors.
                Then 
                $\det\big(\mt{P}_{j_1,\ldots,j_n}^\tr \rmep{\eig}\big) = f_{j_1,\ldots,j_n}(\eig)$.
                Since $\sol{\eig}$ is algebraically simple, the rank of the Jacobian~\eqref{part_der_minors} is full.
                Therefore, we can select $k$ matrices $\mt{P}_1,\ldots,\mt{P}_k$ of the above form that correspond to $k$ linearly independent rows of~\eqref{part_der_minors} at $\eig = \sol{\eig}$.
                When we multiply~\eqref{rectmep} by $\mt{P}_1^\tr, \ldots, \mt{P}_k^\tr$, we obtain a polynomial sMEP
                \begin{equation}
                    \label{rectmepprojB}
                    \mt{P}_i^\tr \rmep{\eig} \rvec_i = \bigg(\sum_{\left\lbrace\bomega\right\rbrace} {\eig}^{\bomega} \mt{P}_i^\tr \mt{A}_{\bomega} \bigg) \rvec_i = \vc{0}, \quad i = 1, \ldots, k,
                \end{equation}
                such that $\sol{\eig} = (\sol{\lambda_1},\ldots,\sol{\lambda_k})$ is an eigenvalue of~\eqref{rectmepprojB} and $\sol{\rvec} \otimes \cdots \otimes \sol{\rvec}$ is the corresponding eigenvector.
                Note that the eigenvalues of
                \eqref{rectmepprojB} are the common roots of
                \begin{equation}
                    \label{eq:det_Pi_M}
                    \det\big(\mt{P}_i^\tr \rmep{\eig}\big) = \vc{0},\quad  i = 1, \ldots, k,
                \end{equation}
                which is a subset of characteristic equations~\eqref{eq:eig_system_minors}.
                Therefore, the polynomial sMEP~\eqref{rectmepprojB} has additional eigenvalues that are not eigenvalues of the rMEP $\rmep{\eig}$.
                Since we have selected $\mt{P}_1,\ldots,\mt{P}_k$ from $k$ linearly independent rows of~\eqref{part_der_minors} at $\sol{\eig}$, the Jacobian of \eqref{eq:det_Pi_M} is nonsingular at $\sol{\eig}$ which is thus
                an algebraically simple eigenvalue of the polynomial sMEP~\eqref{rectmepprojB}.

                Let $\lvec_1 \otimes \cdots \otimes \lvec_k$ be a left eigenvector of $\sol{\eig}$ for
                \eqref{rectmepprojB}.
                Because $\sol{\eig}$ is algebraically simple,
                \begin{equation}
                    \label{MEP_B0_condition}
                    \left[\begin{matrix}
                        \lvec_1^\herm \mt{P}_1^\tr \frac{\partial {\bm{\mathcal A}}}{\partial \lambda_1}({\sol{\eig}}) \sol{\rvec} & \cdots & \lvec_1^\herm \mt{P}_1^\tr \frac{\partial {\bm{\mathcal A}}}{\partial \lambda_k} ({\sol{\eig}})\sol{\rvec} \\
                        \vdots & & \vdots \\
                        \lvec_k^\herm \mt{P}_k^\tr \frac{\partial {\bm{\mathcal A}}}{\partial \lambda_1}({\sol{\eig}}) \sol{\rvec} & \cdots & \lvec_k^\herm \mt{P}_k^\tr \frac{\partial {\bm{\mathcal A}}}{\partial \lambda_k}({\sol{\eig}}) \sol{\rvec}
                    \end{matrix}\right]
                \end{equation}
                is nonsingular. This follows from \cite[Proposition 3.2]{MP_SingMEP} for $k = 2$ and can be generalized to an arbitrary number of parameters in a straightforward way.

                Since $\lvec_i^\herm \mt{P}_i^\tr \rmep{\sol{\eig}} = \vc{0}$ for $i=1,\ldots,k$, vectors $\vc{z}_1 = \mt{P}_1 \lvec_1, \ldots, \vc{z}_k = \mt{P}_k \lvec_k$ belong to $\ker(\rmep{\sol{\eig}}^\herm)$.
                Clearly, the matrix~\eqref{genB0_condition} is singular if vectors $\vc{z}_1, \ldots, \vc{z}_k$ are linearly dependent, therefore, as the matrices in~\eqref{MEP_B0_condition} and~\eqref{genB0_condition} are equal and nonsingular, vectors $\vc{z}_1, \ldots, \vc{z}_k$ have to be linearly independent.
                This proves the theorem for a particular choice of $\vc{z}_1, \ldots, \vc{z}_k$
                and it is easy to see that the theorem then holds for any basis $\vc{z}_1, \ldots, \vc{z}_k$ of $\ker(\rmep{\sol{\eig}}^\herm)$.
            \end{proof}

            \Cref{thm:polysimpleeig} generalizes the well-known result
            which states that if $\sol{\lvec}$ and $\sol{\rvec}$ are the left and the right eigenvectors of a polynomial eigenvalue problem $\mcl{P}(\lambda) \rvec = \vc{0}$ for a simple eigenvalue $\sol{\lambda}$, then $(\sol{\lvec})^\herm \mcl{P}'(\sol{\lambda}) \rvec \ne 0$, see, e.g.,~\cite[Theorem 4.5]{LancasterLambdaMatrices}.

            \begin{corollary}
                \label{thm:simpleeig}
                Let $\sol{\eig}\in\CC^k$ be an algebraically simple eigenvalue of a regular linear rMEP.
                If $\sol{\rvec}$ is the right eigenvector and $\vc{z}_1, \ldots, \vc{z}_k$ are linearly independent vectors from $\ker(\rmep{\sol{\eig}}^\herm)$, then the matrix
                \begin{equation}
                    \label{genB0_conditionB}
                    \begin{bmatrix}
                        \vc{z}_1^\herm \rmat_1 \sol{\rvec} & \cdots & \vc{z}_1^\herm \rmat_k \sol{\rvec} \\
                        \vdots & & \vdots \\
                        \vc{z}_k^\herm \rmat_1 \sol{\rvec} & \cdots & \vc{z}_k^\herm \rmat_k \sol{\rvec}
                    \end{bmatrix}
                \end{equation}
                is nonsingular.
            \end{corollary}

    \section{Homotopy method for linear rMEPs}
        \label{sec:linear}

        Before developing a dedicated homotopy method for rMEPs, we want to point out that, because of the randomized sketching approach in~\cite{HKP_NLAA_2023_RectMEP}, all numerical methods for sMEPs can be applied to rMEPs as well.
        This also applies to the existing homotopy methods for sMEPs, such as the methods developed in~\cite{DongYuYu_SIMAX, Rodriguez_Fiber_Homotopy}.

        Suppose that we have a linear rMEP
        \begin{equation}
            \label{linrMEP}
            \rmep{\eig} \rvec = \left(\rmat_0 + \lambda_1 \rmat_1 + \cdots + \lambda_k \rmat_k\right) \rvec = \vc{0},
        \end{equation}
        where $\rmat_i \in \CC^{(n + k - 1) \times n}$ for $i = 0, \ldots, k$.
        As shown in~\cite[Theorem 1]{HKP_NLAA_2023_RectMEP}, if we multiply~\eqref{linrMEP} by random matrices $\mt{P}_1, \ldots, \mt{P}_k\in \CC^{n\times (n + k - 1)}$, we obtain a sMEP
        \begin{equation}
            \label{rectmepproj}
            \smep{\eig} \rvec_i = \left(\mt{P}_i \rmat_0 + \lambda_1 \mt{P}_i \rmat_1 + \ldots + \lambda_k \mt{P}_i \rmat_k\right) \rvec_i = \vc{0}, \quad i = 1, \ldots, k,
        \end{equation}
        such that if $(\eig,\rvec)$ is an eigenpair of~\eqref{linrMEP},
        then $(\eig,\rvec \otimes \cdots \otimes \rvec)$ is an
        eigenpair of~\eqref{rectmepproj}.
        So, when we compute all eigenvalues of~\eqref{rectmepproj}, for example, by using a homotopy method for sMEPs, we can extract the eigenvalues that are solutions of~\eqref{linrMEP}.

        The drawback of the above-described approach is that the projected problem~\eqref{rectmepproj} has $n^k$ eigenvalues, of which only $\binom{n+k-1}{k}$ are eigenvalues of~\eqref{linrMEP}.
        When $n^k\gg \binom{n+k-1}{k}$, this becomes highly inefficient.
        A homotopy method must first track $n^k$ paths to compute the eigenvalues of the sMEP,
        and then identify the eigenvalues of~\eqref{linrMEP} by evaluating the residuals.
        The discrepancy in size can be huge; for instance, for $n = 7$ and $k = 6$ problem~\eqref{rectmepproj} has $\num{117649}$ eigenvalues but only $\num{924}$ of them are  eigenvalues of~\eqref{linrMEP}. Clearly, it would be much more efficient to track only these $\num{924}$ paths.
        This motivates the development of a homotopy method tailored specifically to rMEPs.

        The common idea of all homotopy methods, see, e.g.,~\cite{Allgower-90-Continuation-Methods}, is to track solution paths of a continuously parameterized problem.
        Suppose that we want to solve a nonlinear equation $f(z) = 0$ that has $m$ solutions.
        We construct the initial nonlinear equation $g(z) = 0$ of the same type with the same number of solutions in such a way that we know its solutions $z_1, \ldots, z_m$.
        Next, we construct a homotopy $h(t, z)$ that depends smoothly on the parameter $t \in [0, 1]$, so that $h(0, z) = g(z)$ and $h(1, z) = f(z)$, e.g., 
        $h(t, z) \coloneq (1 - t) g(z) + t f(z)$.
        Assuming that there exist $m$ disjoint paths $z_i(t)$ such that $h(t, z_i(t)) = 0$ for $t \in [0,1]$, $z_i(0) = z_i$, we track them from $t = 0$ to $t = 1$ to obtain the solutions of $f(z) = 0$.

        To apply a homotopy method to the linear rMEP~\eqref{linrMEP}, we need to be able to construct an initial rMEP with the same structure and number of solutions that we can solve efficiently.
        Afterwards, the continuation procedure can track the different eigenpairs to retrieve the solutions of the original rMEP. We discuss this below.

        \subsection{Construction of an initial problem with known solutions}

            We construct the initial rMEP as
            \begin{equation}
                \label{initial}
                \irmep{\eig} \rvec = \left(\irmat_0 + \lambda_1 \irmat_1 + \cdots + \lambda_k \irmat_k\right) \rvec = \vc{0},
            \end{equation}
            where $\irmat_0$ is a banded matrix with random entries on the main diagonal and the next $k - 1$ subdiagonals, i.e., $[\irmat_0]_{ij} = 0$ for $i < j$ or $i - j > k -1$, and $\irmat_i$ are shifted identity matrices, i.e.,
            \begin{equation}
                \label{matBi}
                \irmat_i = \begin{bmatrix}
                    \mt{0}_{(i - 1) \times n} \\
                    \mt{I}_n \\
                    \mt{0}_{(k - i) \times n}
                \end{bmatrix}, \quad i = 1, \ldots, k,
            \end{equation}
            where $\mt{0}_{m \times n}$ and $\mt{I}_n$ are $m \times n$ zero matrices and $n \times n$ identity matrices, respectively.
            By construction, rMEPs of the above form are finitely regular and generic, see~\cite{ShapiroShapiro}.
            Therefore, we can expect that~\eqref{initial} has $\binom{n+k-1}{k}$ simple eigenvalues.
            We now explain how we can retrieve these solutions.

            \subsubsection{Two parameters}

                For $k = 2$ we can simply write down the eigenvalues and eigenvectors of~\eqref{initial}.
                Denote the elements on the $i$th diagonal of $\irmat_0$ as $\elem{i}{1}, \ldots, \elem{i}{n}$.
                The initial problem is given by
                \begin{equation}
                    \label{eq:twopolynomialmatrix}
                    \irmat_0 + \lambda_1 \irmat_1 + \lambda_2 \irmat_2 = \begin{bmatrix}
                        \elem{0}{1} + \lambda_1 & 0 & \cdots & 0 \\
                        \elem{1}{1} + \lambda_2 & \elem{0}{2} + \lambda_1 & \ddots & \vdots \\
                        0 & \ddots & \ddots & 0 \\
                        \vdots & \ddots & \elem{1}{n - 1} + \lambda_2 & \elem{0}{n} +\lambda_1 \\
                        0 & \cdots & 0 & \elem{1}{n} + \lambda_2
                    \end{bmatrix}.
                \end{equation}
                The rank of
                \eqref{eq:twopolynomialmatrix} is below $n$ if we set a
                diagonal element and a subdiagonal element in a lower row to zero.
                Then, it is easy to see that the eigenvalues of~\eqref{initial} are
                \[
                    (\lambda_1, \lambda_2) = \big(-\elem{0}{j}, -\elem{1}{\ell}\big), \quad 1 \le j \le \ell \le n,
                \]
                which gives the $\binom{n+1}{2}$ eigenvalues that we are looking for.
                Notice that all eigenvalues are simple for generic elements of $\irmat_0$, i.e., for generic $\elem{j}{i}$, $i = 1, \ldots, n$ and $j =0, 1$.

            \subsubsection{Three parameters}

                For $k = 3$ we have to solve multiple standard (one-parameter) eigenvalue problems.
                The initial problem is now given by $\irmat_0 + \lambda_1 \irmat_1 + \lambda_2 \irmat_2 + \lambda_3 \irmat_3 =$
                \begin{equation}
                    \label{inihomo3}
                    \begin{bmatrix}
                        \elem{0}{1} + \lambda_1 & 0 & 0 & \cdots & 0 \\
                        \elem{1}{1} + \lambda_2 & \elem{0}{2} +\lambda_1 & 0 & \cdots & 0 \\
                        \elem{2}{1} + \lambda_3 & \elem{1}{2} + \lambda_2 & \elem{0}{3} +\lambda_1 & \ddots & \vdots \\
                        0 & \ddots & \ddots & \ddots & 0 \\
                        0 & \ddots & \elem{2}{n - 2} + \lambda_3 & \elem{1}{n - 1} + \lambda_2 & \elem{0}{n} + \lambda_1 \\
                        \vdots & \ddots & 0 & \elem{2}{n - 1} + \lambda_3 & \elem{1}{n} +\lambda_2 \\
                        0 & \cdots & 0 & 0 & \elem{2}{n} + \lambda_3
                    \end{bmatrix}.
                \end{equation}
                Similarly as for $k = 2$, we set $\lambda_1 = -\elem{0}{j}$ and $\lambda_3 = -\elem{2}{\ell}$ for
                $1 \le j\le\ell \le n$ to zero out one element on the upper diagonal and one on the lower diagonal from different rows in~\eqref{inihomo3}.
                The rank of \eqref{inihomo3} is then smaller than $n$ if the submatrix
                with rows $j + 1, \ldots, \ell+1$ and columns $j, \ldots, \ell$ is singular, i.e., when
                \[
                    \det\left(\begin{bmatrix}
                        \elem{1}{j} & \elem{0}{j + 1} - \elem{0}{j} & 0 & \cdots & 0 \\
                        \elem{2}{j} - \elem{2}{l} & \elem{1}{j + 1} & \elem{0}{j + 2} - \elem{0}{j} & \ddots & \vdots \\
                        0 & \ddots & \ddots & \ddots & 0 \\
                        \vdots & \ddots & \elem{2}{l - 2} - \elem{2}{l} & \elem{1}{\ell - 1} & \elem{0}{\ell} - \elem{0}{j} \\
                        0 & \cdots & 0 & \elem{2}{l-1} - \elem{2}{l} & \elem{1}{\ell}
                    \end{bmatrix} + \lambda_2 \mt{I}_{\ell - j+1}\right) = 0.
                \]
                By solving this square standard eigenvalue problem, we get $\ell - j +1$ values of $\lambda_2$ for each pair
                $(\lambda_1, \lambda_3) = \big(-\elem{0}{j}, -\elem{2}{\ell}\big)$ for $1 \le j \le \ell \le n$.
                Consequently, the obtained tuples $(\lambda_1, \lambda_2, \lambda_3)$ correspond to all $\binom{n+2}{3}$ eigenvalues of the initial problem $\irmep{\eig}$.

            \subsubsection{More than three parameters}

                In general, when $k \ge 4$, we apply a similar approach as for $k = 3$, but with the difference that we need to solve several smaller $(k - 2)$-parameter linear rMEPs to find the numerical values of $\lambda_2, \ldots, \lambda_{k - 1}$.
                We provide the construction in \cref{alg:initialproblem}.
                For each $\lambda_1 = -\elem{0}{j}$ and $\lambda_k = -\elem{k - 1}{\ell}$ with $1 \le j \le \ell \le n$, we construct a 
                submatrix of~\eqref{initial} composed of rows $j + 1, \ldots, \ell + k - 2$
                and columns $j, \ldots, \ell$.
                Each of the $(k - 2)$-parameter linear rMEPs has $r = \binom{\ell-j+k-2}{k-2}$ solutions.

                \begin{algorithm}[h]
                    \caption{Construct an initial linear rMEP with known solutions}
                    \label{alg:initialproblem}
                    \begin{algorithmic}[1]
                        \Require $n$: number of columns, $k$: number of parameters
                        \Ensure matrices $\irmat_0, \ldots, \irmat_k$ and eigenpairs
                        $\left(\eig^{(j)}, \rvec^{(j)}\right)$, $j=1,\ldots,\binom{n+k-1}{k}$
                        \State Construct banded matrix $\irmat_0$ with random entries $\elem{j}{i}$, $i = 1, \ldots, n$, $j = 0, \ldots, k - 1$
                        \State Construct shifted identity matrices $\irmat_1, \ldots, \irmat_k$ as in~\eqref{matBi}
                        \For{$j = 1, \ldots, n$}
                            \For{$\ell = j, \ldots, n$}
                                \State $\mttilde{B}_i = \irmat_i(j + 1:\ell + k - 2, j:\ell)$ for $i = 0, \ldots, k$
                                \State Compute eigenpairs $\left(\vc{\mu}^{(p)}, \vc{z}^{(p)}\right)$, $p = 1, \ldots, r$, of the smaller rMEP
                                \[
                                    \left(\mttilde{B}_0 - \elem{0}{j} \mttilde{B}_1 - \elem{k - 1}{\ell} \mttilde{B}_k +\sum_{i = 2}^{k - 1} \mu_i \mttilde{B}_i\right) \vc{z} = \vc{0}
                                \]
                                \State Set $\eig^{(p)} = \left(-\elem{0}{j}, \vc{\mu}^{(p)}, -\elem{k - 1}{\ell}\right)$ and $\rvec^{(p)} = \begin{bmatrix} \vc{0}_{j-1} & {\vc{z}^{(p)}}^\tr & \vc{0}_{n-\ell} \end{bmatrix}^\tr$
                                \State Add eigenpairs $\left(\eig^{(p)}, \rvec^{(p)}\right)$, $p = 1, \ldots, r$, to the list of solutions
                            \EndFor
                        \EndFor
                        \State \Return $\irmat_0, \ldots, \irmat_k$ and $\left(\eig^{(j)}, \rvec^{(j)}\right)$, $j=1,\ldots,\binom{n+k-1}{k}$
                    \end{algorithmic}
                \end{algorithm}

                Although the structure of the initial problem admits an efficient construction of the associated joint GEP in~\eqref{deltamep}, a recursive solution approach via a homotopy method is for most combinations of $n$ and $k$ the most efficient.
                Applying this recursively until $k < 4$ yields the solutions of the initial problem.

        \subsection{Path continuation procedure}

            For the homotopy, we consider the homogeneous version of the original rMEP in~\eqref{linrMEP},
            \begin{equation}
                \label{linrMEPhomo}
                \rmep{\heig} \rvec = \left(\lambda_0 \rmat_0 + \lambda_1 \rmat_1 + \cdots + \lambda_k \rmat_k\right) \rvec = \vc{0},
            \end{equation}
            in which $\heig = (\lambda_0, \lambda_1, \ldots, \lambda_k)$ are coordinates in the projective space and thus $\heig \ne \vc{0}$.
            In the homogeneous setting, eigenvalues of \eqref{linrMEPhomo} with $\lambda_0 = 0$ correspond to infinite eigenvalues of \eqref{linrMEP} and each finite eigenvalue $(\lambda_1, \ldots, \lambda_k)$ of \eqref{linrMEP} corresponds to the eigenvalue $(1, \lambda_1, \ldots, \lambda_k)$ of \eqref{linrMEPhomo}.
            We transform the initial problem~\eqref{initial} in a similar way into a homogeneous linear rMEP,
            \begin{equation}
                \label{initialhomo}
                \irmep{\heig} \rvec = \left(\lambda_0 \irmat_0 + \lambda_1 \irmat_1 + \cdots + \lambda_k \irmat_k\right) \vc{x} = \vc{0}.
            \end{equation}

            Now we
            define $\mt{H}_i(t) = (1 - t) \gamma \irmat_i + t \rmat_i$ for $i = 0, \ldots, k$, where $\gamma\ne 0$ is a random complex number, let $\mcl{H}(t,\heig) = \sum_{i = 0}^k \lambda_i \mt{H}_i(t)$,
            and construct the homotopy
            as
            \begin{equation}
                \label{homoABC}
                \homo{t} = \begin{bmatrix}
                        \mcl{H}(t,\heig) \rvec \\
                        \heig^\herm \heig - 1 \\
                        \rvec^\herm \rvec - 1
                    \end{bmatrix} = \begin{bmatrix}
                    \displaystyle{\sum_{i = 0}^k \lambda_i \left((1 - t) \gamma \irmat_i + t \rmat_i\right) \rvec} \\
                    \heig^\herm \heig - 1 \\
                    \rvec^\herm \rvec - 1
                \end{bmatrix}.
            \end{equation}
            The introduction of $\gamma$ is a standard
            approach in the homotopy methods
            that prevents path intersections, see, e.g., \cite{DongYuYu_SIMAX} and Section \ref{sec:homotopy_details}.
            The last two rows in~\eqref{homoABC} normalize the eigenvector $\rvec$ and the homogeneous eigenvalue $\heig$, which both have to be nonzero.
            Clearly, $\mcl{H}(0,\heig) = \gamma \irmep{\heig}$ and
            $\mcl{H}(1,\heig) =  \rmep{\heig}$.
            As $t$ moves from $0$ to $1$, the initial rMEP~\eqref{initialhomo} smoothly changes to the  original
            rMEP~\eqref{linrMEPhomo} that we want to solve.

        We outline the entire homotopy method in \cref{alg:homotopy} and give additional details below. For a visual summary, see Figs.~\ref{fig:Example1_Intro} and
        \ref{fig:homotopy} in Section \ref{sec:numericalexamples}.

        \begin{algorithm}[h]
            \caption{Solve a linear rMEP via homotopy}
            \label{alg:homotopy}
            \begin{algorithmic}[1]
                \Require $\rmat_0, \ldots, \rmat_k$: coefficient matrices of $\rmep{\eig}$
                \Ensure $\left(\eig^{(j)}, \rvec^{(j)}\right)$ for $j=1,\ldots,\binom{n+k-1}{k}$: eigenpairs
                \State Construct initial linear rMEP $\irmep{\eig}$ with known solutions (\cref{alg:initialproblem})
                \State Homogenize $\rmep{\eig}$ and $\irmep{\eig}$, and construct homotopy $\homo{t}$ as in~\eqref{homoABC}
                \For{each solution $\left(\eig^{(j)}, \rvec^{(j)}\right)$ of $\irmep{\eig} \rvec=\vc{0}$}
                    \State Set $t = 0$, $\eig(0) = \eig^{(j)}$, and $\rvec(0) = \rvec^{(j)}$
                    \State Choose stepsize $h$
                    \While{$t < 1$}
                        \State Predict next point: $\eig_\mathrm{P} = \eig(t) + h \dot{\eig}(t)$, $\rvec_\mathrm{P} = \rvec(t) + h \dot{\rvec}(t)$ as in~\eqref{eq:prediction}
                        \State Correct with  Newton's method: $\eig_\mathrm{C} = \eig_\mathrm{P} + \Delta \eig$, $\rvec_\mathrm{C} = \rvec_\mathrm{P} + \Delta \rvec$ as in~\eqref{eq:correction}
                        \State Set $t = t + h$
                        \State Update $\eig(t) = \eig_\mathrm{C}$ and $\rvec(t) = \rvec_\mathrm{C}$
                    \EndWhile
                    \State Store $\left(\eig(1), \rvec(1)\right)$ as $\left(\eig^{(j)}, \rvec^{(j)}\right)$
                \EndFor
                \State \Return $\left(\eig^{(j)}, \rvec^{(j)}\right)$ for $j=1,\ldots,\binom{n+k-1}{k}$
            \end{algorithmic}
        \end{algorithm}

            \subsubsection{Predictor-corrector scheme}

                Let $(\heig^\circ, \rvec^\circ)$ be a normalized eigenpair of~\eqref{initialhomo}, i.e., $\|\heig^\circ\|_2 = \norm{\rvec^\circ}_2 = 1$.
                We assume that there exists a smooth path $(\heig(t), \rvec(t))$ such that
                $\heig(0) = \heig^\circ$, $\rvec(0) = \rvec^\circ$, and
                $H(t,\heig(t), \rvec(t))=\vc{0}$
                for $0\le t \le 1$.
                We track it  to $t = 1$ by the standard
                predictor-corrector scheme, see, e.g., \cite{DongYuYu_SIMAX, P_ContRD2EP}, where
                we
                use Euler's method and Newton's method.

                First, we use Euler's method to predict the next point on the path, which is at $t + h$, with $h$ the stepsize.
                By differentiating $\homo{t} = \vc{0}$, we get the linear system for the tangent vector
                \begin{equation}
                    \label{tangent_homo}
                    J(t, \heig, \rvec) \begin{bmatrix}
                        \dot{\lambda}_0 \\ \vdots \\ \dot{\lambda}_k \\ \dot{\rvec}
                    \end{bmatrix} = \begin{bmatrix}
                        \displaystyle{\sum_{i = 0}^k \lambda_i \left(\gamma \irmat_i - \rmat_i\right) \rvec} \\ 0 \\ 0
                    \end{bmatrix},
                \end{equation}
                where the corresponding Jacobian matrix $J(t, \heig, \rvec)$ is
                \begin{equation}
                    \label{jacF}
                    \begin{bmatrix}
                        \left((1 - t) \gamma\irmat_0 + t \rmat_0\right) \rvec & \cdots & \left((1 - t) \gamma \irmat_k + t \rmat_k\right) \rvec & \displaystyle{\sum_{i = 0}^k \lambda_i \left((1 - t) \gamma \irmat_i + t \rmat_i\right)} \\
                        2 \overline{\lambda}_0 & \cdots & 2 \overline{\lambda}_k & 0 \\
                        0 & \cdots & 0 & 2 \rvec^\herm
                    \end{bmatrix}.
                \end{equation}
                Note that the last two rows of \eqref{homoABC} are not complex differentiable. However, it is known in the context of optimization problems that we can employ the approximation \eqref{jacF};
                for the details, see \cite{LuSu_NLAA}.
                We solve the linear system~\eqref{tangent_homo} and take
                \begin{equation}
                    \label{eq:prediction}
                    \heig_\mathrm{P} = \heig + h \begin{bmatrix} \dot{\lambda}_0 & \cdots & \dot{\lambda}_k \end{bmatrix}^\tr, \quad \rvec_\mathrm{P} = \rvec + h \dot{\rvec},
                \end{equation}
                for the predictor step.

                Next, we use Newton's method with $(\heig_\mathrm{P}, \rvec_\mathrm{P})$ as initial guess to compute the solution of~\eqref{homoABC} at $t + h$.
                In one step of Newton's method, we update $(\heig^{(r)}, \rvec^{(r)})$ to $(\heig^{(r + 1)}, \rvec^{(r + 1)})$ as
                $
                    \heig^{(r + 1)} = \heig^{(r)}+ [\Delta \lambda_0^{(r)}\ \cdots\ \Delta \lambda_k^{(r)} ]^\tr$, $\rvec^{(r + 1)} = \rvec^{(r)} + \Delta \rvec^{(r)},
                $
                where the updates are components of the solution of the linear system
                \begin{equation}
                    \label{NewtonStep}
                    J(t + h, \heig^{(r)}, \rvec^{(r)}) \begin{bmatrix}
                        \Delta \lambda_0^{(r)} \\ \vdots \\ \Delta \lambda_k^{(r)} \\[2pt] \Delta \rvec^{(r)}
                    \end{bmatrix} = - H(t + h, \heig^{(r)}, \rvec^{(r)}).
                \end{equation}
                The different Newton updates lead to a corrected eigenpair:
                \begin{equation}
                    \label{eq:correction}
                    \eig_\mathrm{C} = \eig_\mathrm{P} + \Delta \eig, \quad \rvec_\mathrm{C} = \rvec_\mathrm{P} + \Delta \rvec.
                \end{equation}

            \subsubsection{Singularities along the path}

             The following theorem shows that the singular points of the homotopy method are the points where $\heig$ is a multiple eigenvalue of $\mcl{H}(t, \heig)$.

                \begin{theorem}
                    If $\heig$ is a simple eigenvalue of $\mcl{H}(t, \heig)$ for $t \in [0, 1)$ and $\rvec$ is the corresponding eigenvector, then the Jacobian matrix $J(t, \heig, \rvec)$ is nonsingular.
                \end{theorem}

                \begin{proof}
                    Take $\bmu = \begin{bmatrix} \mu_0 & \cdots & \mu_k \end{bmatrix}^\tr \in \CC^{k + 1}$ and $\vc{z} \in \CC^n$ such that
                    \begin{equation}
                        \label{eq:Jac_nonsing}
                        J(t, \heig, \rvec) \begin{bmatrix} \bmu \\ \vc{z} \end{bmatrix} = \vc{0}.
                    \end{equation}
                    Then,
                    \begin{equation}
                        \label{prf1}
                        \mcl{H}(t, \heig) \vc{z} + \mu_0 \mt{H}_0(t) \rvec + \cdots + \mu_k \mt{H}_k(t) \rvec = \vc{0}.
                    \end{equation}
                    For linearly independent vectors $\lvec_1, \ldots, \lvec_k$ from $\ker(\mcl{H}(t, \heig)^\herm)$, we get from~\eqref{prf1} that
                    \begin{equation}
                        \label{prf2}
                       \begin{bmatrix}
                           \lvec_1^\herm \mt{H}_0(t) \rvec & \cdots & \lvec_1^\herm \mt{H}_k(t) \rvec \\
                           \vdots & & \vdots \\
                           \lvec_k^\herm \mt{H}_0(t) \rvec & \cdots & \lvec_k^\herm \mt{H}_k(t) \rvec
                       \end{bmatrix} \begin{bmatrix}
                           \mu_0 \\ \vdots \\ \mu_k
                       \end{bmatrix} = \vc{0}.
                    \end{equation}
                    It follows from a homogeneous version of Corollary \ref{thm:simpleeig} that the $k \times (k + 1)$ matrix in~\eqref{prf2} has full rank, so its null space is one-dimensional.
                    Therefore, $\bmu$ and $\heig$ are collinear and $\bmu =\alpha \heig$ for a scalar $\alpha$.
                    But, by considering~\eqref{jacF} and~\eqref{eq:Jac_nonsing}, $2\heig^\herm \bmu = 0$, so $\alpha = 0$ and $\bmu = 0$.
                    From~\eqref{prf1} it now follows that $\vc{z} = \beta \rvec$ for a scalar $\beta$, because $\mcl{H}(t, \heig)$ also has a one-dimensional null space.
                    Similarly, it follows from $2 \rvec^\herm \vc{z} = 0$ that $\beta = 0$.
                    So, since the only solution for~\eqref{eq:Jac_nonsing} is $\bmu = \vc{0}$ and $\vc{z} = \vc{0}$, this
                    means that $J(t, \heig, \rvec)$ is nonsingular.
                \end{proof}

                As there can be only finitely many points $t \in [0,1)$ such that
                 $\mcl{H}(t, \heig)$ has a multiple eigenvalue, we
                can show that for a random $\gamma\in\CC$ the homotopy paths do not
                contain any singular points, for details see Section \ref{sec:homotopy_details}.
                If $\rmep{\heig}$ has a multiple eigenvalue of multiplicity $r$, then $r$ homotopy paths collide at $t = 1$.

                \begin{remark}
                    If $\gamma$ and all matrices of $\rmep{\heig}$ and $\irmep{\heig}$ are real, then for each $t\in[0,1]$
                    the eigenvalues of $\mcl{H}(t, \heig)$ are either real or appear in complex conjugate pairs.
                    If the numbers of real eigenvalues at $t=0$ and $t=1$ do not agree,
                    there must
                    exist transition points $t\in(0,1)$, where two real eigenvalue curves transform into a complex conjugate pair or vice versa.
                    At each such bifurcation point we have a multiple (generically double) eigenvalue $\heig$ and the Jacobian $J(t, \heig, \rvec)$ is singular.

                    Thus, even if the original rMEP has real matrices $\rmat_0, \ldots, \rmat_k$, it is vital to select a complex random $\gamma$ and complex random entries of the matrix $\irmat_0$ to avoid these difficulties. By doing this, we have to follow the paths
                    using more expensive complex arithmetic, but the advantage is that generically all eigenvalues of $\mcl{H}(t, \heig)$ for $t\in[0,1)$ are complex, simple, and
                    do not appear in complex conjugate pairs.

                    We remark that for a real rMEP $\rmep{\heig}$ it should be possible to use only real arithmetic by
                    applying a different tracking strategy, where we use arc length as a parameter and detect bifurcation points to switch between complex and real curves,
                    in a similar way as it has been done in \cite{LiZengCong_1992} for a nonsymmetric eigenvalue problem and in \cite{BP_ContWE2EP} for a weakly elliptic two-parameter sMEP.
                \end{remark}

            \subsubsection{Step size and path jumping}

                As we follow the paths numerically, jumps between paths can occur, e.g., if the stepsize $h$ is too large or if the Jacobian matrix is close to singular. A possible solution is to retrace all paths that lead to a cluster of $r>1$ eigenvalues such that
                the corresponding eigenvectors are not linearly independent, using a smaller stepsize $h$
                and other settings that make the computation more accurate. Since
                finer steps require more time, it makes sense
                to use them only for the paths that have to be retraced.

    \section{Homotopy method for polynomial problems}
        \label{sec:polynomial}

        Compared to a linear rMEP with the same number of parameters $k$ and columns $n$, a generic polynomial rMEP of degree $d$ has $d^k$ times as many eigenvalues.
        Even if the number of eigenvalues of the linear rMEP is such that the problem is solvable with the existing methods in~\cite{HKP_NLAA_2023_RectMEP} and \cite{VermeerschDeMoorLAA}, the factor $d^k$ in~\eqref{number_of_solutions} makes the problem quickly too large for increasing $d$.
        In the polynomial setting, the need for a homotopy method is even more pronounced.
        We start by looking at the generic case (\cref{sec:genericpolynomial}), but also mention special non-generic polynomial structures (\cref{sec:nongenericpolynomial}).
        For both cases, the previously developed linear homotopy plays a crucial role.

        \subsection{Standard generic structures}
            \label{sec:genericpolynomial}

            Suppose that for $n$ and $k$ we can construct an initial linear rMEP
            \begin{equation}
                \label{initialP}
                \irmep{\vc{\mu}} \rvec = \left(\irmat_0 + \mu_1 \irmat_1 + \cdots + \mu_k \irmat_k
                \right) \rvec = \vc{0},
            \end{equation}
            e.g., with \cref{alg:initialproblem}.
            Applying the substitution $\mu_i =  \lambda_i^d$ for $i = 1, \ldots, k$ we then obtain an initial polynomial rMEP of degree $d$ of the form
            \begin{equation}
                \label{initialPK}
                 \irmep{\eig} \rvec = \left(\irmat_0 + \lambda_1^d \irmat_1 + \cdots + \lambda_k^d \irmat_k\right) \rvec = \vc{0}.
            \end{equation}
            Clearly, if $(\mu_1, \ldots, \mu_k)$ is an eigenvalue of~\eqref{initialP}, then $(\lambda_1, \ldots, \lambda_k) = (\mu_1^{1/d}, \ldots, \mu_k^{1/d})$ is an eigenvalue of~\eqref{initialPK}.
            Since we get $d$ different $d$th roots for each component, from each~eigenvalue of \eqref{initialP} we get $d^k$ different eigenvalues of~\eqref{initialPK} that
            share the same eigenvector $\rvec$.
            The number of solutions then agrees with the generic number in~\eqref{number_of_solutions}.

            Consider the $d$th degree version of~\eqref{rectmep}.
            Its homogeneous form is given by
            \begin{equation}
                \label{rectmep_homo}
                \rmep{\heig} \rvec \coloneq \bigg(\sum_{|\bomega| = d} {\heig}^{\bomega} \rmat_{\bomega} \bigg) \rvec = \vc{0},
            \end{equation}
            where $\heig = (\lambda_0, \lambda_1, \ldots, \lambda_k)$ and the multi-index is now $\bomega = (\omega_0, \omega_1, \ldots, \omega_k)$.
            The sum considers all multi-indices $\bomega$ with degree $|\bomega|=\omega_0 + \omega_1 + \cdots + \omega_k$ equal to $d$.
            Similarly, the homogeneous version of~\eqref{initialPK} is equal to
            \begin{equation}
                \label{init_PK}
                \irmep{\heig} \rvec \coloneq \bigg(\sum_{i=0}^k \lambda_i^d \irmat_i \bigg) \rvec = \vc{0}.
            \end{equation}
            Notice that the initial rMEP \eqref{init_PK} only has univariate monomials.
            Now, same as in the linear case, we define $\mcl{H}(t, \heig) = (1 - t) \gamma \irmep{\heig} + t \rmep{\heig}$ and use the homotopy
            \begin{equation}
                \label{homoABC_Poly}
                \homo{t} = \begin{bmatrix}
                    \mcl{H}(t, \heig) \rvec \\
                    \heig^\herm \heig - 1 \\
                    \rvec^\herm \rvec - 1
                \end{bmatrix}.
            \end{equation}
            The corresponding Jacobian matrix is
            \begin{equation}
                \label{jacF_poly}
                J(t, \heig, \rvec) = \begin{bmatrix}
                    \displaystyle{\frac{\partial \mcl{H}}{\partial \lambda_0}(t, \heig) \rvec} & \cdots & \displaystyle{\frac{\partial \mcl{H}}{\partial \lambda_k}(t, \heig) \rvec} & \mcl{H}(t, \heig) \\[2ex]
                    2 \overline{\lambda}_0 & \cdots & 2 \overline{\lambda}_k & 0 \\
                    0 & \cdots & 0 & 2 \rvec^\herm
                \end{bmatrix}
            \end{equation}
            and for the tangent vector we have to solve
            \[
                J(t, \heig, \rvec) \begin{bmatrix}
                    \dot{\lambda}_0 \\ \vdots \\ \dot{\lambda}_k \\ \dot{\rvec}
                \end{bmatrix} = \begin{bmatrix} \gamma \irmep{\heig} \rvec - \rmep{\heig} \rvec \\ 0 \\ 0 \end{bmatrix}.
            \]
            Again, we use Euler's predictor followed by Newton's method to compute the next point on the path.
            In each step, we solve the same linear system as in~\eqref{NewtonStep}, but now with~\eqref{jacF_poly} and~\eqref{homoABC_Poly}.

            So, as long as we can construct and solve the initial linear rMEP, we can apply the homotopy method
            to solve the polynomial rMEP.
            In many situations, when $\binom{n+k-1}{k}$ is manageable but $d^k \binom{n+k-1}{k}$ is too large, \eqref{initialPK} provides a numerical approach that can compute all the solutions for large degrees.
            This adapted homotopy method is, to the best of our knowledge, the only numerical method capable of doing this at the moment.
            The path following part may be slow, as there can be many paths to follow, but the memory requirements are not restrictive. Furthermore, the paths can be followed independently in parallel, which is so embarrassingly simple to implement that using  $M$ workers  we can expect a speedup close to a factor of $M$.

        \subsection{Special non-generic structures}
            \label{sec:nongenericpolynomial}

            A polynomial rMEP of degree $d$
            can have fewer than~\eqref{number_of_solutions} affine eigenvalues when the structure is non-generic and not all $d$th degree monomials are present.
            For example, while a generic $3$-parameter rMEP
            of degree $2$ has $8 \binom{n+2}{3}$ eigenvalues, an rMEP 
            \begin{equation}
                \label{spec_case3}
                \left(\rmat_{000} + \lambda_1 \rmat_{100} + \lambda_2 \rmat_{010} + \lambda_3 \rmat_{001} + \lambda_1^2 \rmat_{200}\right) \rvec = \vc{0}
            \end{equation}
            with generic coefficient matrices has only $2 \binom{n+2}{3}$ solutions.
            We can exploit the structure and reduce the number of paths to follow.
            For the particular case~\eqref{spec_case3}, we could use an initial rMEP
            of the form
            \begin{equation}
                \label{init_spec_case3}
                \left( \irmat_0 + \lambda_1^2 \irmat_1 + \lambda_2 \irmat_2 + \lambda_3 \irmat_3\right) \rvec = \vc{0},
            \end{equation}
            which has the same number of solutions as~\eqref{spec_case3}.
            We construct~\eqref{init_spec_case3} from a linear rMEP~\eqref{initialP} in a similar way as in the generic case
            by using the substitutions $\mu_1 = \lambda_1^2$, $\mu_2 = \lambda_2$, and $\mu_3 = \lambda_3$.

    \section{Numerical examples}
        \label{sec:numericalexamples}

        All numerical experiments were carried out in \textsc{Matlab} 2025b on a desktop computer with an i7–11700K 3.6 GHz CPU and 64 GB RAM.
        Eight workers were used to track the paths in parallel.
        In the recursive construction of the initial rMEP, methods from MultiParEig \cite{multipareig_29} and MacaulayLab \cite{macaulaylab} are used to solve the smaller rMEPs when this is estimated to be more efficient than the homotopy method.
        The \textsc{Matlab} implementation of the homotopy
        method, together with all numerical examples, is available in the public repository \url{https://github.com/borplestenjak/RMEP_homotopy}.

        \subsection{Zero-dimensional solution sets}

            \begin{example}
                A linear two-parameter rMEP
                \begin{equation}\label{eq:intro_example}
                    \left(\begin{bmatrix}
                        -1 & \phantom{-}1 \\
                        \phantom{-}2 & -2 \\
                        \phantom{-}3 & -1
                    \end{bmatrix} + \lambda_1 \begin{bmatrix}
                        \phantom{-}1 & \phantom{-}2 \\
                        -2 & \phantom{-}3 \\
                        \phantom{-}3 & \phantom{-}3 \\
                    \end{bmatrix} + \lambda_2 \begin{bmatrix}
                        \phantom{-}1 & \phantom{-}1 \\
                        \phantom{-}2 & -2 \\
                        -3 & \phantom{-}4
                    \end{bmatrix}\right) \rvec = \vc{0}
                \end{equation}
                has three simple eigenvalues
                $\eig^{(1)}=(2.29390,   1.62430)$,
                $\eig^{(2)}=(0.18703,  -0.61192)$, and
                $\eig^{(3)}=(-0.44292,   0.70259)$.

    The homotopy method starts from a problem of the form~\eqref{eq:twopolynomialmatrix} with generically three simple eigenvalues, follows three paths, and finds all eigenvalues.
    For example, if we use
    \begin{equation}\label{eq:intro_B0A}
       \mcl{B}_1(\eig)=\begin{bmatrix}
            2+i & 0 \\
            3-i & 1+i \\
            0 & 2-i
        \end{bmatrix}
        + \lambda_1 \begin{bmatrix}
            1 & 0 \\
            0 & 1 \\
            0 & 0
        \end{bmatrix}
        + \lambda_2 \begin{bmatrix}
            0 & 0 \\
            1 & 0 \\
            0 & 1
        \end{bmatrix}
    \end{equation}
    and take $\gamma=1$,
                then the paths that start in the eigenvalues
                $\mu^{(1)}=(-2-i, -3 +i)$,
                $\mu^{(2)}=(-1-i,   -2+i)$, and
                $\mu^{(3)}=(-2-i,  -2+i)$
                of  \eqref{eq:intro_B0A}
                lead, respectively, to $\eig^{(1)}$, $\eig^{(2)}$, and $\eig^{(3)}$.
            We provide a graphical presentation of the homotopy paths (only the affine part of the eigenvalues) in the left-hand side of Fig.~\ref{fig:Example1_Intro}.
            The paths are
            clearly smooth and suitable for tracing by Algorithm~\ref{alg:homotopy}.

            Although the eigenvalues of \eqref{eq:intro_example} are real, it is important to use a complex initial problem.
            For example, if we instead use a real initial rMEP,
                \begin{equation}\label{eq:intro_B0B}
                    \mcl{B}_2(\eig)=\begin{bmatrix}
                        \phantom{-}2 & 0 \\
                        -1 & 1 \\
                        \phantom{-}0 & 2
                    \end{bmatrix}
                    + \lambda_1 \begin{bmatrix}
                        1 & 0 \\
                        0 & 1 \\
                        0 & 0
                    \end{bmatrix}
                    + \lambda_2 \begin{bmatrix}
                        0 & 0 \\
                        1 & 0 \\
                        0 & 1
                    \end{bmatrix},
                \end{equation}
                then from the paths starting in the eigenvalues
                $\mu^{(1)}=(-2, 1)$,
                $\mu^{(2)}=(-1,   -2)$, and
                $\mu^{(3)}=(-2,  -2)$
                of  \eqref{eq:intro_B0B}
                one leads to the eigenvalue $\eig^{(1)}$, while
                the other two fail to pass
                a bifurcation point  $(-1.09413,  -0.12213)$ at $t=0.49386$, where
                a real eigenvalue pair transforms into a complex conjugate pair.
                The right-hand side of Fig.~\ref{fig:Example1_Intro} visualizes this behavior.


 \begin{figure}[h]
    \includegraphics[width=7.2cm]{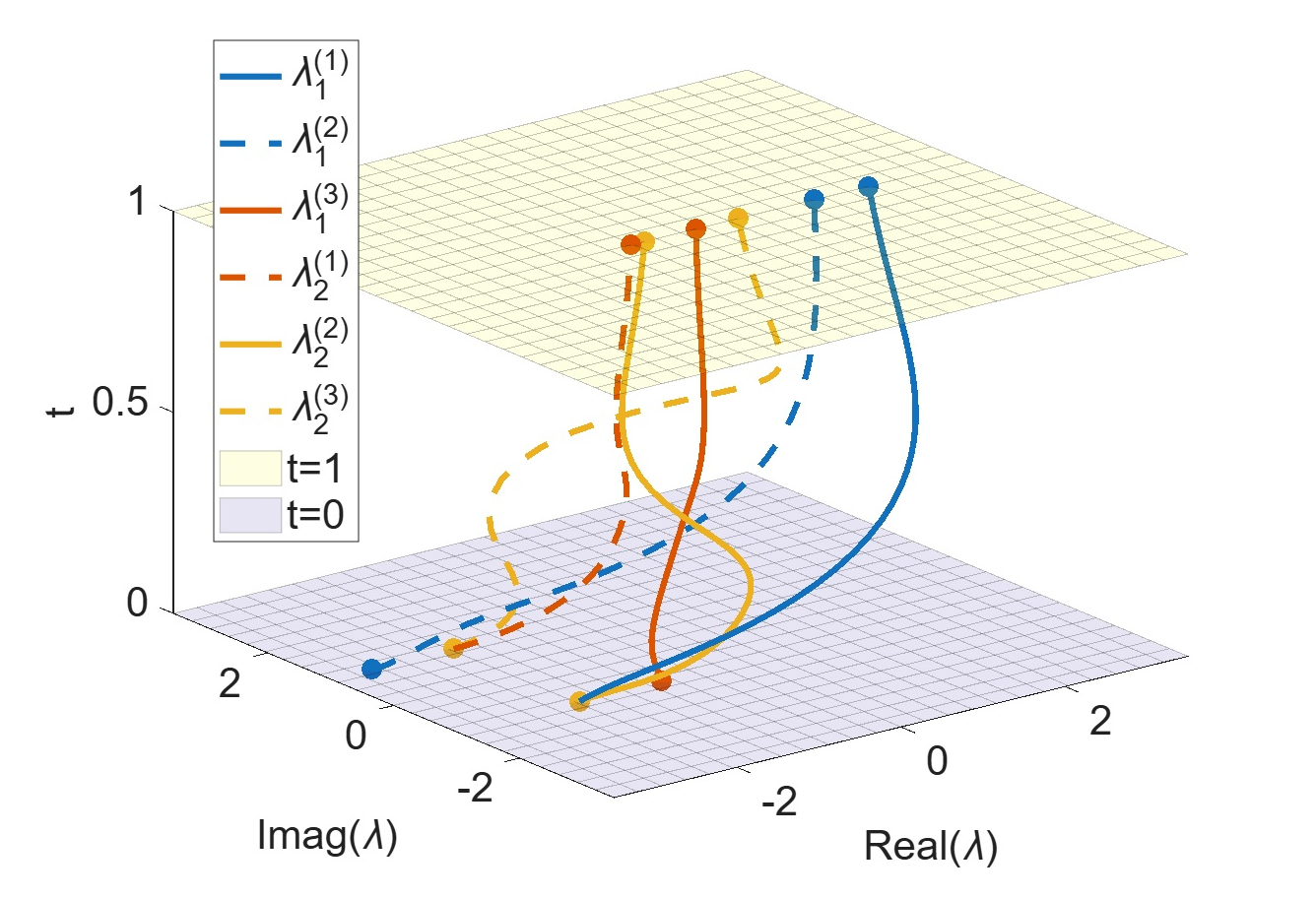}\hspace{-0.5cm}
    \includegraphics[width=6.2cm]{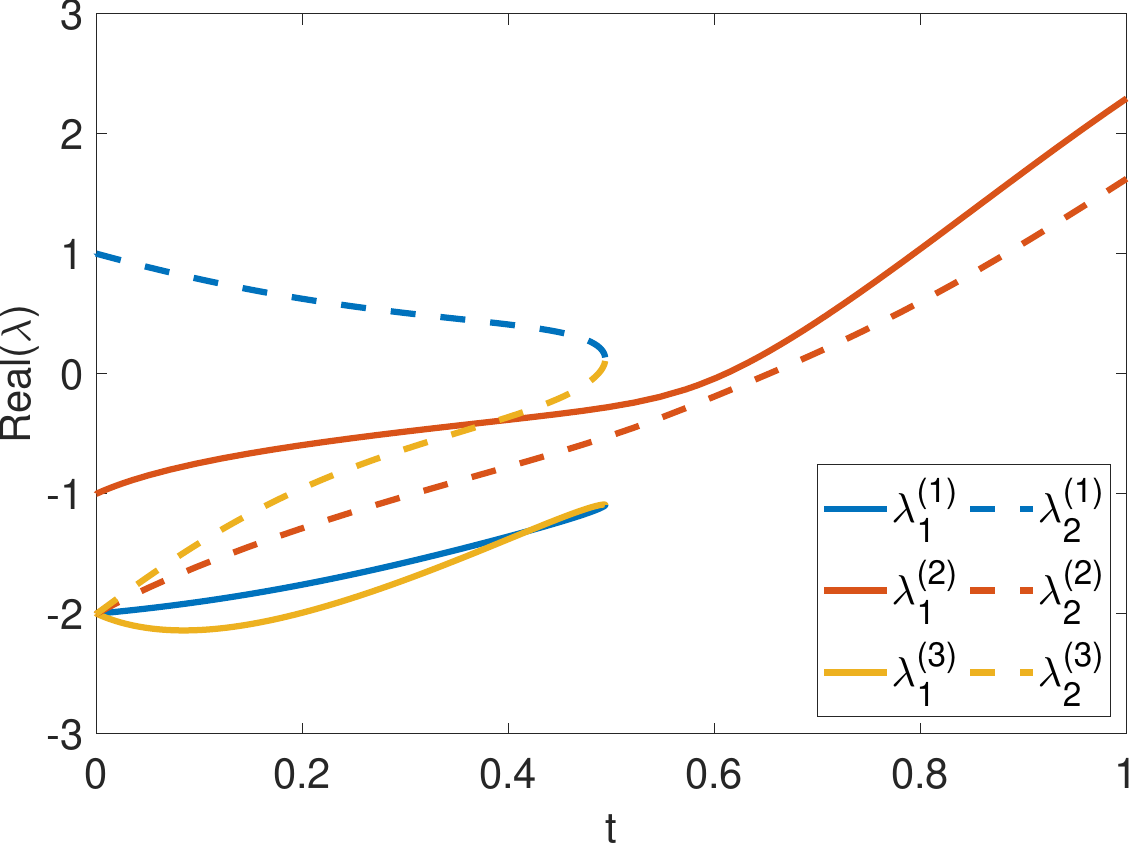}
    \caption{Homotopy paths (affine eigenvalue part only) for the rMEP \eqref{eq:intro_example}. Left: if we start with the complex initial problem \eqref{eq:intro_B0A} then complex curves lead to real eigenvalues of \eqref{eq:intro_example}. Right: if we start
    from the real initial problem \eqref{eq:intro_B0B} then paths are real and some get stuck at a
    bifurcation point.
    \label{fig:Example1_Intro}}
\end{figure}
\end{example}

            \begin{example}
                \label{ex:linear}
                We test the homotopy method on random linear rMEPs with matrices whose
                entries are independent real standard normal variables. Such generic problems are regular
                and have finite simple eigenvalues.

                Based on extensive numerical experiments, we selected
                the entries $\elem{j}{i}$, $i = 1, \ldots, n$, $j = 0, \ldots, k - 1$,
                of the banded matrix $\irmat_0$ of the initial rMEP \eqref{initialP}
                in a way that the entries on each of the $k$ diagonals form
                a randomly ordered and randomly rotated set of $n$-th roots of unity.
                In our tests, this choice consistently produced among
                the best-conditioned Jacobian matrices
                $J(0, \heig^\circ, \rvec^\circ)$ over all
                 normalized eigenpairs $(\heig^\circ, \rvec^\circ)$
                of~\eqref{initialhomo}. We provide an explanation for this choice in Section \ref{sec:optchoice}.

                We compared the performance and accuracy of the homotopy method
                to \verb|rect_multipareig| from \cite{multipareig_29} and \verb|macaulaylab| from \cite{macaulaylab}.
                A comparison of computation times to compute all eigenvalues
                is given in Fig.~\ref{fig:Example1_k}.
                The left-hand side shows how the computation times for random rMEPs with
                a fixed number of parameters $k$ 
                depend
                on the number of columns $n$, while the right-hand  side shows how
                the computation times for random rMEPs with a fixed number of columns $n$ 
                depend on the number of parameters $k$.

                Table \ref{tab:ex1} contains the computation time for selected combinations of $n$ and $k$.
                We also provide timing results of the sequential version of Algorithm~\ref{alg:homotopy} and the computation time required for the construction of the initial rMEP. Note
                that for problems with many parameters the initial construction can be more expensive than the path following phase.
                We see that in the problems with many paths, the speedup in the parallel version is close to a factor eight, the number of workers.
                For large values of $k$ and $n$, the
                sequential version is already faster than the other two methods, which both
                eventually fail for rMEPs with
                too many eigenvalues.

 \begin{figure}[h]
    \includegraphics[width=6.4cm]{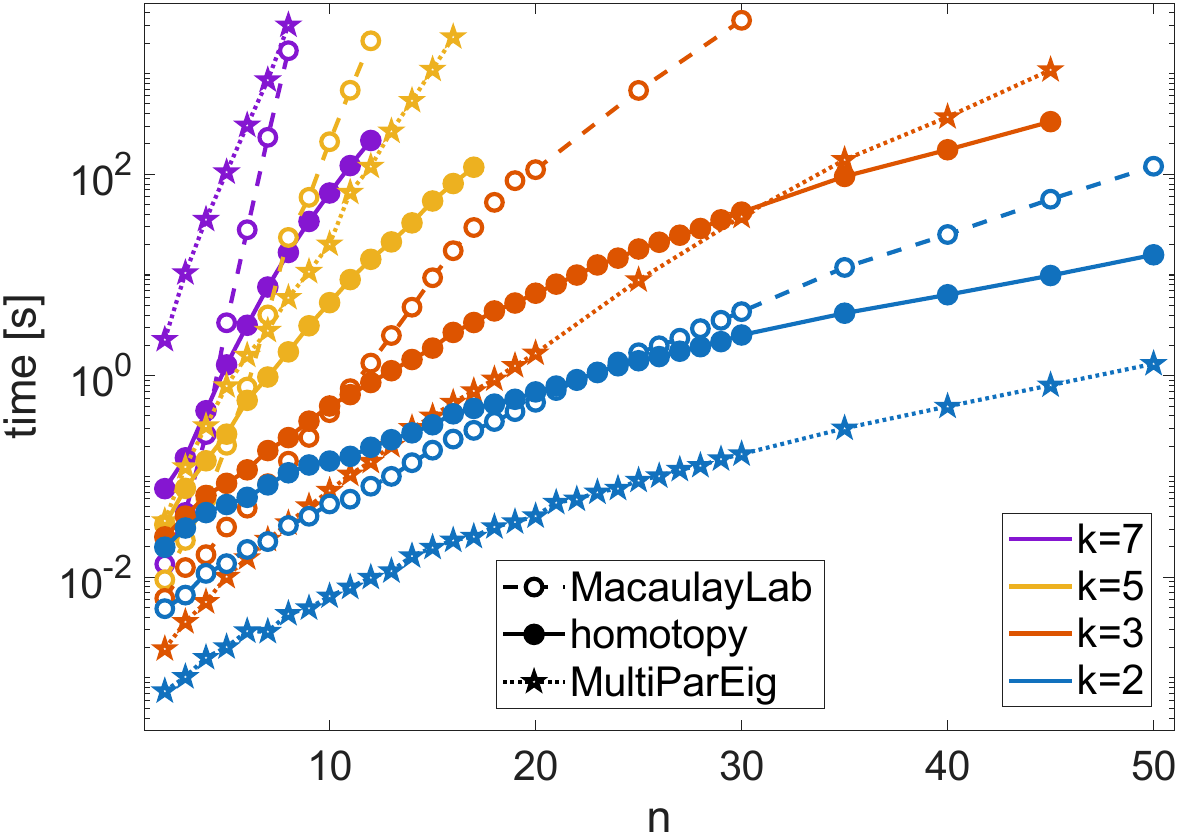}\
    \includegraphics[width=6.4cm]{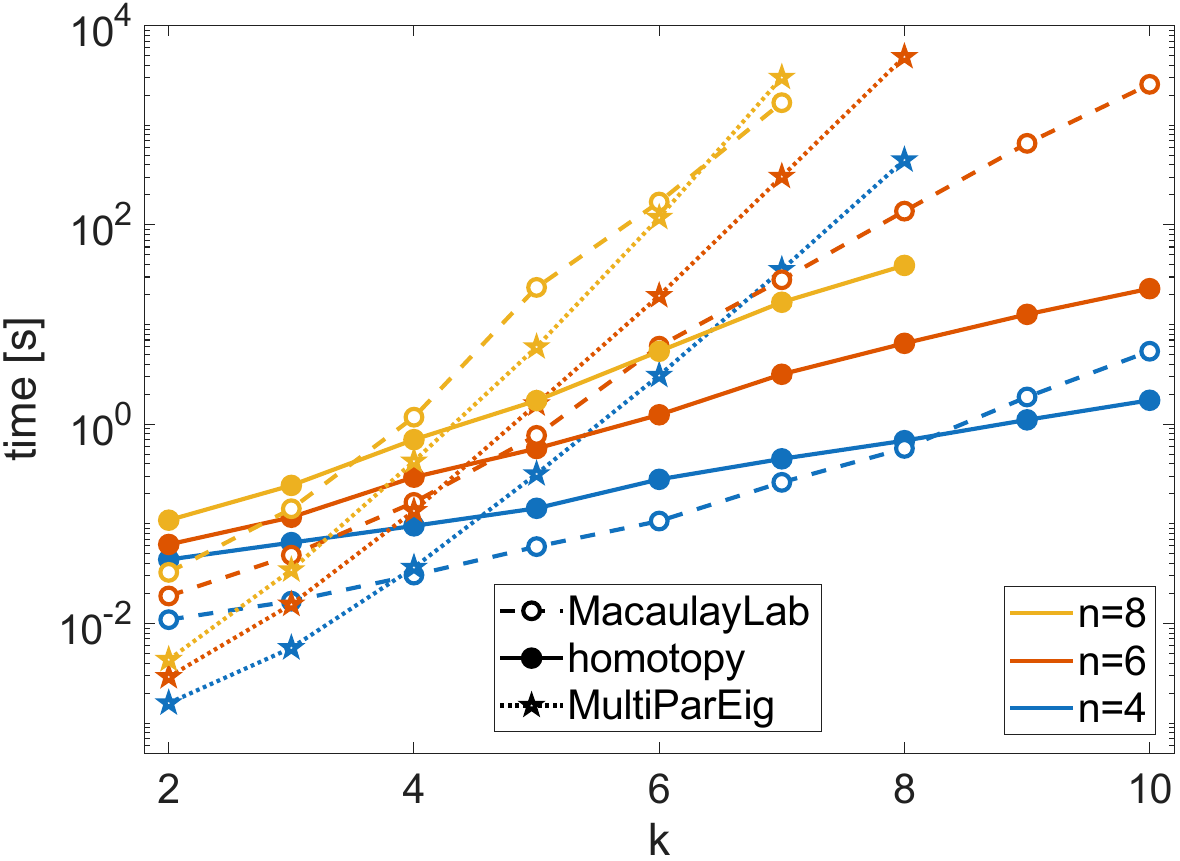}
    \caption{Computation times to solve a random linear rMEP with $k$ parameters and $n$ columns by
    the homotopy method, \texttt{macaulaylab} and \texttt{rect\_multipareig}. Left: dependence on $n$ for fixed $k$, Right: dependence on $k$ for fixed $n$.
    \label{fig:Example1_k}}
\end{figure}

\begin{table}[htb!]
\centering
\caption{Computation times in seconds of
    the homotopy method (sequential and parallel version with 8 workers), \texttt{macaulaylab} and \texttt{rect\_multipareig} to solve
    a random linear rMEP for selected problems with $n$ columns and $k$ parameters.\label{tab:ex1}}
\footnotesize
\begin{tabular}{rrr|r|r|rr|rrr} \hline \rule{0pt}{2.6ex}%
$k$ & $n$ & eigs & MacaulayLab & MultiParEig & Hom.(S) & init & Hom.(P) & init & speedup\\[0.5mm]
\hline \rule{0pt}{2.6ex}%
  2 &  25 & 325&   $1.7$ & $\mathbf{0.09}$ & $7.2$ & $0.02$ & $1.4$ & $0.02$ & $5.1$\\
  2 &  50 & 1275&  $120.1$ & $\mathbf{1.3}$ & $ 104.6$ & $0.09$ & $15.9$  & $0.09$ & $6.6$\\
  3 &  25 & 2925&  $679.6$ & $\mathbf{8.9}$ & $109.9$ & $0.4$ & $18.1$ & $0.4$ & $6.1$ \\
  3 &  40 & 11480&  out of mem. & $370.4$ & $1119.4$ & $2.6$ & $\mathbf{175.1}$ & $2.6$ & $6.4$\\
  4 &  15 & 3060&  $636.9$ & $14.4$ & $63.8$ & $0.9$ & $\mathbf{11.1}$ & $0.9$ & $5.7$\\
  5 &  15 & 11628&  out of mem. & $1097.6$ & $306.9$& $9.7$ & $\mathbf{54.3}$  & $9.7$ & $5.7$ \\
  8 &  6 & 1287&  $137.5$ & $4870.7$ &  $20.8$ & $2.4$ & $\mathbf{5.4}$ & $2.2$ & $3.9$\\
  10 &  6 & 3003&  $2574.2$ & out of mem. & $80.1$ & $9.4$ & $\mathbf{21.3}$ & $8.7$ & $3.8$\\
  10 &  10 & 92378&  out of mem. & out of mem. & 4943.9 & 2703.4 & $\mathbf{1116.2}$ & $718.6$ & $4.4$\\
  20 &  3 & 231&  $13.8$ & out of mem. & $15.0$ & $10.0$ & $\mathbf{5.6}$ & $4.6$ & $2.7$\\
  60 &  2 & 61&  $64.6$ & out of mem. & $36.8$ & $31.4$ & $\mathbf{8.5}$ & $7.5$ & $4.3$\\
\hline
\end{tabular}
\end{table}

 The homotopy method is the fastest approach
 if $n$ or $k$ are large enough. The method from MultiParEig is the most efficient approach for problems with $k=2$ or $k=3$ and a moderate number of columns, while the method using block Macaulay matrices
 is the best choice for problems with a small number of columns and $k\ge 4$ parameters.

 The comparison of the maximal residuals in Fig.~\ref{fig:Example1_err} shows that the homotopy method is the most accurate method,
 which can be explained by the use of Newton's method in the path following step. We remark that we could refine the results obtained by the other two methods in a similar way by adding Newton's method.

 \begin{figure}[h]
    \includegraphics[width=6.3cm]{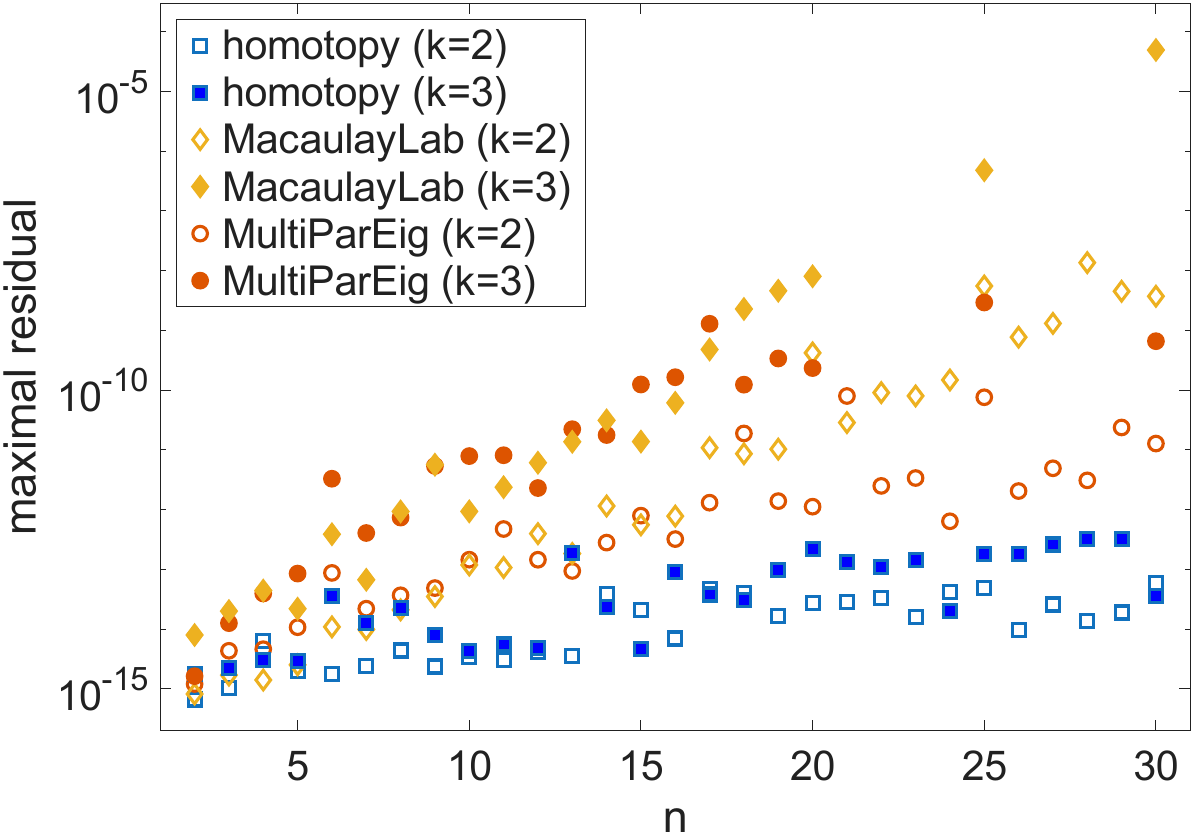}\ \ \
    \includegraphics[width=6.3cm]{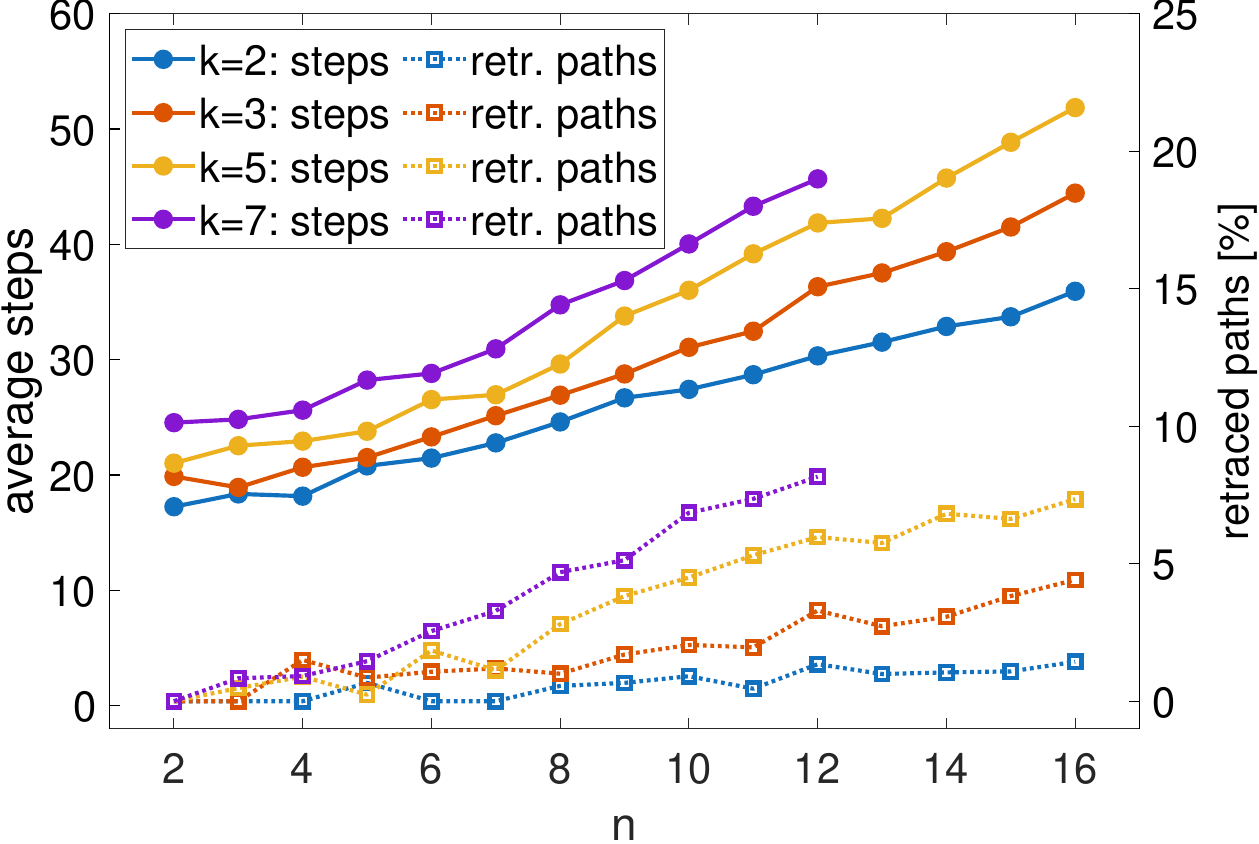}
    \caption{Results for various numbers of spectral parameters. Left: maximal residual for random linear rMEPs with $k=2,3$ parameters for
    the homotopy method, \texttt{macaulaylab} and \texttt{rect\_multipareig}.
    Right: Average number of steps for an eigenvalue path (solid line) and percentage of retraced paths (dotted line) for the homotopy method for linear rMEPs with $k=2,3,5,7$ parameters.\label{fig:Example1_err}}
\end{figure}

 The right-hand side of Fig.~\ref{fig:Example1_err} contains two plots per number of parameters: The solid line presents the average number of steps for following one path in the homotopy method and shows that the number of steps grows linearly with the number of columns $n$, while the dotted line presents the average percentage of paths that were retraced with more accurate settings.
 We see that the percentage of retracing slightly grows with $n$ and $k$.

\end{example}

 \begin{example}
        In a similar way we applied the homotopy method to random polynomial rMEPs with matrices whose entries are real standard normal variables. As in Example~\ref{ex:linear}, such problems are regular
        and have finite simple eigenvalues.
        We compared the computation times
        to methods \verb|macaulaylab| from \cite{macaulaylab} and
        \verb|rect_quad_multipareig| from \cite{multipareig_29} (the latter is
        implemented only for quadratic rMEPs with $2$, $3$, and $4$ parameters).
        We report computation times for random quadratic rMEPs with
        $2$, $3$ and $4$ parameters in the left-hand side of Fig.~\ref{fig:Example2_k}, while
        in the right-hand side we compare the homotopy method to \verb|macaulaylab|
        for random polynomial rMEPs with two parameters ($k=2$) of degree $3$ to $7$.

        In Table \ref{tab:ex2}, we consider selected combinations of the
        degree $d$, number of parameters $k$, and number of columns $n$.
        Again, we notice that the homotopy method is the fastest method for large $n$ or $k$.
        We can see that for the simplest polynomial case, a two-parameter quadratic rMEP, the homotopy method becomes the fastest method for $n \ge 60$.

\begin{figure}[h]
    \includegraphics[width=6.4cm]{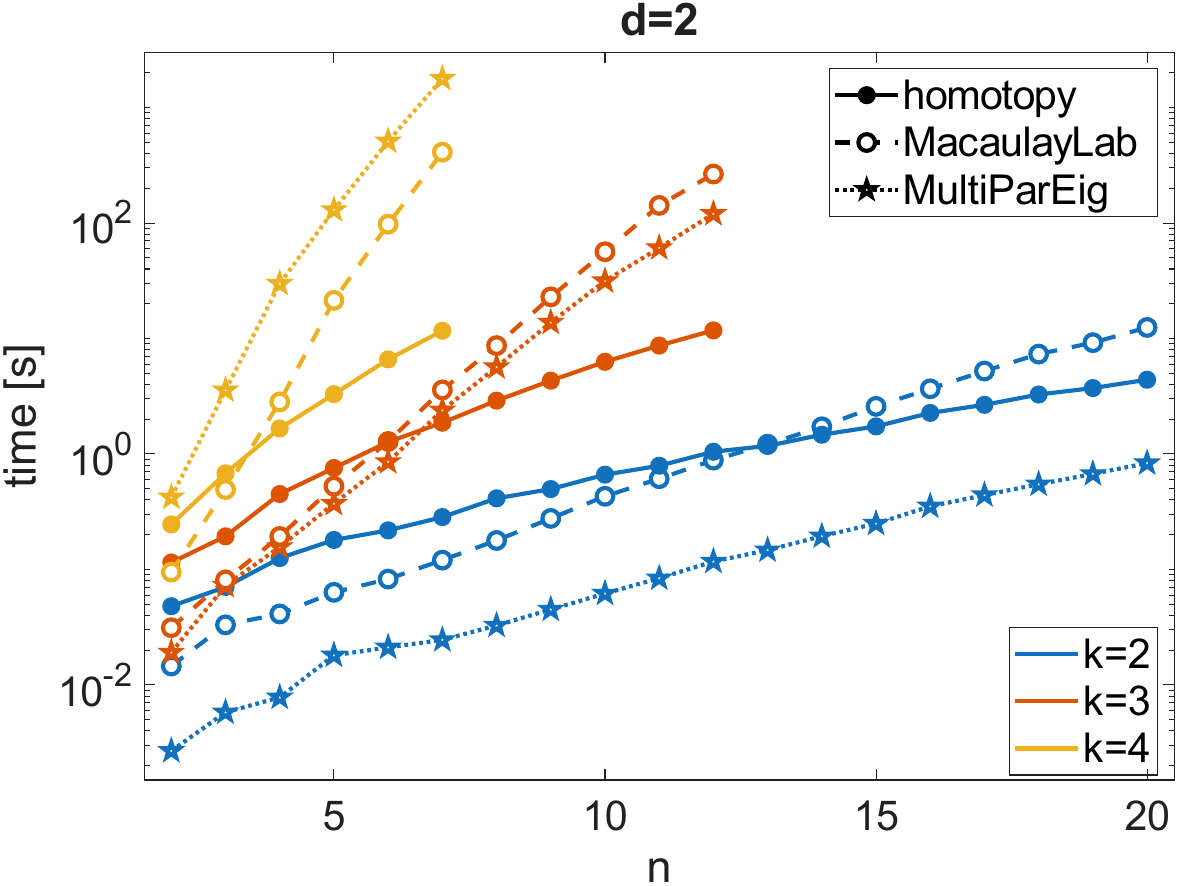}\
    \includegraphics[width=6.4cm]{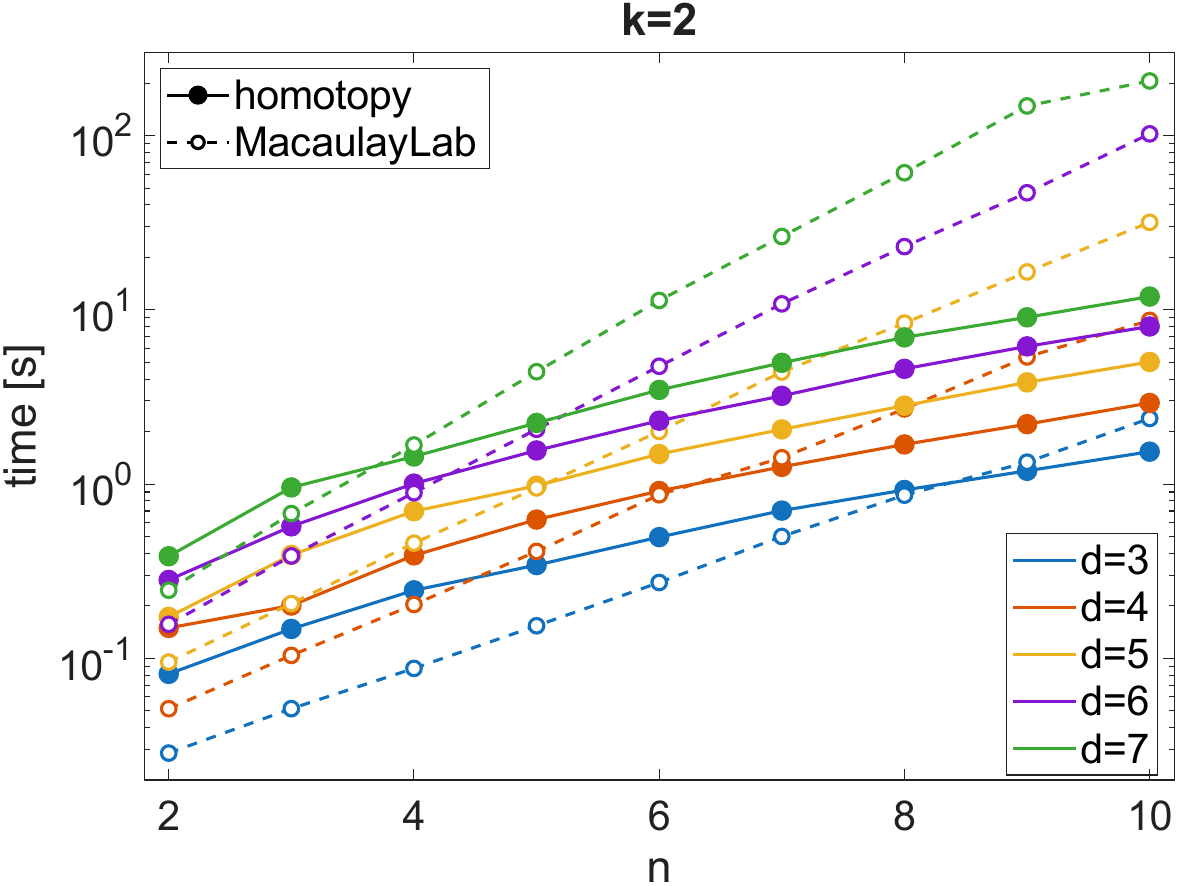}
    \caption{Results for various degrees. Left: Computation times for random quadratic rMEPs with $k=2,3,4$ parameters
    using the homotopy method, \texttt{macaulaylab}, and \texttt{rect\_quad\_multipareig}. Right: Times for random polynomial two-parameter rMEPs of degree $d=3,\ldots,7$ using the homotopy method and \texttt{macaulaylab}.\label{fig:Example2_k}}
\end{figure}

\begin{table}[htb!]
\centering
\caption{Computation times in seconds for
    the homotopy method, \texttt{macaulaylab} and \texttt{rect\_quad\_multipareig} to solve
    a random polynomial rMEP for selected combinations of degree $d$, $n$ columns and $k$ parameters.\label{tab:ex2}}
 \footnotesize
\begin{tabular}{rrrr|r|r|rr} \hline \rule{0pt}{2.6ex}%
$d$ & $k$ & $n$ & eigs & MacaulayLab & MultiParEig & Homotopy & init \\[0.5mm]
\hline \rule{0pt}{2.6ex}%
  2 & 2 &  30 & 1860&  $178.5$ & $7.0$ & $16.4$ & $0.04$ \\
  2 & 2 &  60 & 7320&  out of memory & $284.7$ & $238.4$ & $0.24$ \\
  2 & 3 &  12 & 2912&  $265.8$ & $120.2$ & $11.8$ & $0.06$ \\
  2 & 4 &  7 &  3360&  $410.9$ & $1781.4$ & $11.7$ & $0.07$\\
  3 & 2 &  15 & 1080&  $20.1$ & not available & $4.3$& $0.01$  \\
  4 & 3 &  5  & 2240&  $78.5$ & not available &  $8.1$ & $0.07$ \\
  7 & 2 &  10 & 2695&  $206.0$ & not available & $11.9$ & $0.04$\\
  10 & 2 & 6 & 2100&  $106.6$ & not available & $10.1$ & $0.03$ \\
\hline
\end{tabular}
\end{table}
\end{example}

        \subsection{Positive-dimensional solution sets}

            Although we do not provide a supporting theory, the homotopy method works in principle also for rMEPs with positive-dimensional solution sets.
            As a proof of concept, we give several numerical examples. For such problems, if the Jacobian
            is nonsingular for all $t\in[0,1)$, then some homotopy paths converge to points on a positive-dimensional solution, while the remaining paths converge to the isolated eigenvalues, as illustrated in Fig.~\ref{fig:homotopy}.

            \begin{figure}[h]
                \centering
                \includegraphics[width=10cm]{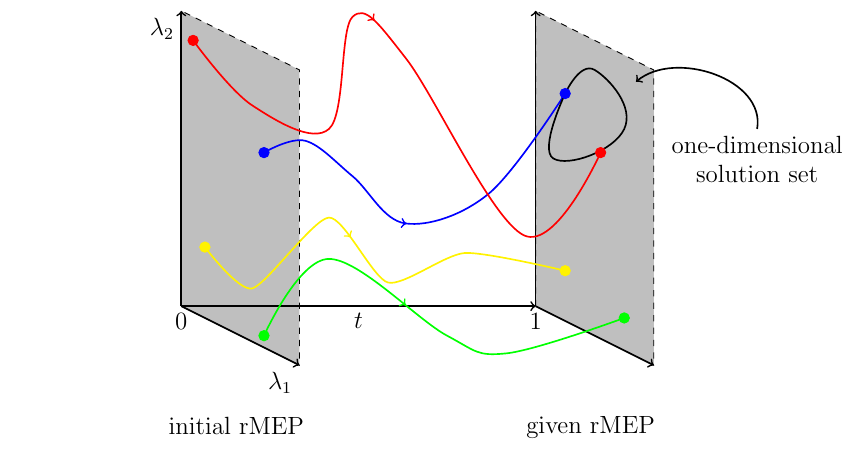}
                \caption{Visual explanation of the homotopy method for a two-parameter eigenvalue problem.
                The homotopy method tracks the solution paths through the continuous deformation from the initial rMEP to the given rMEP, thereby recovering the latter's unknown solutions. In the particular example
                the given rMEP has
                a one-dimensional solution set in addition to
                the simple eigenvalues.}
                \label{fig:homotopy}
            \end{figure}

            The main problem with such rMEPs is that it is very difficult or even impossible to detect the paths that have to
            be retraced due to path jumping.
            We can find paths that converged to the same eigenvalue, but, when a path that should in exact arithmetic converge to an isolated eigenvalue has numerically converged to a point on a positive-dimensional solution, then it appears as a simple eigenvalue and will not be retraced.

            For instance, for many classes of rMEPs with a positive-dimensional solution set at infinity we know
            the expected number of finite eigenvalues.
            If the homotopy method computes fewer finite eigenvalues than expected, then some paths leading to infinite eigenvalues should be retraced.
            But, as we cannot detect individual bad paths, we would have to retrace all paths that converged to infinite eigenvalues. Usually there are so many such eigenvalues that it is not efficient to follow the paths again and we are limited to a single round of path following.

            \begin{example}
                The two-parameter eigenvalue problem
                \begin{equation}
                    \left(\begin{bmatrix}
                        1 & 3 \\
                        2 & 2 \\
                        3 & 1
                    \end{bmatrix} + \lambda_1 \begin{bmatrix}
                        1 & 2 \\
                        2 & 3 \\
                        3 & 1 \\
                    \end{bmatrix} + \lambda_2 \begin{bmatrix}
                        1 & 1 \\
                        2 & 2 \\
                        3 & 0
                    \end{bmatrix}\right) \rvec = \vc{0}
                \end{equation}
                has a one-dimensional solution set
                $\{(\lambda_1,\lambda_2): \lambda_1+\lambda_2=-1\}$ and one eigenvalue
                $\eig^{(1)}=(-4,4)$.
                The homotopy method starts from a generic problem of the form~\eqref{eq:twopolynomialmatrix} that has three simple eigenvalues, follows three paths, and finds the
                eigenvalue $\eig^{(1)}$ and two points on the one-dimensional set.
                Different initial problems lead to different points on the one-dimensional set.
                For example, if we
                start from the initial problem
                \eqref{eq:twopolynomialmatrix}, where we take $\gamma=1$ and use for
                $B_0$ respectively
                \[
                    B_{0A}=\begin{bmatrix}
                        2+i & 0 \\
                        3-i & 1+i \\
                        0 & 2-i
                    \end{bmatrix}\quad\textrm{and}\quad
                    B_{0B}=\begin{bmatrix}
                        3+i & 0 \\
                        4+i & 5-i \\
                        0 & 6-2i
                    \end{bmatrix},
                \]
                we get solutions
                \[\begin{array}{rclrclr}
                \eig^{(1A)}&=&(-4,4),& \kappa_{1A}&=&1.06\cdot 10^1,\\
                \eig^{(2A)}&=&(\phantom{-}1.9505 - 0.6963i, -2.9505 + 0.6963i),& \kappa_{2A}&=&2.68\cdot 10^{16},\\
                \eig^{(3A)}&=&(-2.2840 - 0.3032i,   \phantom{-}1.2840 + 0.3032i),& \kappa_{3A}&=&9.61\cdot 10^{15},
               \end{array}
                \]
                and
                \[\begin{array}{rclrcl}
                \eig^{(1B)}&=&(-4,4),& \kappa_{1B}&=&1.06\cdot 10^1,\\
                \eig^{(2B)}&=&  (-2.6762 - 0.2693i,  \phantom{-}1.6762 + 0.2693i),&\kappa_{2B}&=&2.06\cdot 10^{14},\\
                \eig^{(3B)}&=&  (\phantom{-}2.6761 + 0.6026i,  -3.6761 - 0.6026i),&\kappa_{3B}&=&1.81\cdot 10^{16}.
               \end{array}
                \]
             To identify which of the computed eigenvalues are isolated simple eigenvalues,
             we can apply Corollary \ref{thm:simpleeig} and check the
             nonsingularity of
             the matrix            \eqref{genB0_conditionB}.
             The condition numbers
             $\kappa_{1A},\ldots,\kappa_{3B}$ of the
             corresponding matrix \eqref{genB0_conditionB} clearly show
             that in both cases the homotopy method finds the
             eigenvalue $\eig^{(1)}$.
            \end{example}

            The next two examples come from systems and control applications, which motivates the development of
            homotopy continuation methods, as these rMEPs grow quickly with the number of measured data points or model orders.

            \begin{example}
                 Autoregressive moving-average (ARMA) models represent an observed signal $\vc{y} \in \Rset^N$ as the output of a linear filter driven by an unobserved, latent input signal $\vc{e} \in \Rset^M$.
                 An ARMA$(p,1)$ model is a linear time-invariant model that combines a regression of the observed output $y[k]$ on its previous $p$ values with a linear combination of the latent inputs $e[k]$ and $e[k - 1]$,
                 \begin{equation}\label{eq:arma}
                     y[k] + \sum_{i=1}^p \alpha_i y[k - i] = e[k] + \gamma e[k - 1],\quad k=p+1,\ldots,N,
                 \end{equation}
                 where the weighting coefficients $\alpha_1,\ldots,\alpha_p$ and $\gamma$ are the unknown model parameters.
                 For a given $\vc{y} \in \Rset^N$,
                 we seek the model parameters $(\alpha_1,\ldots,\alpha_p, \gamma)$ that minimize
                 \begin{equation}
                     \begin{gathered}
                         \min_{\alpha_1,\ldots,\alpha_p, \gamma, \vc{e}} \frac{1}{2} \norm{\vc{e}}_2^2,\quad                          \text{subject to } \eqref{eq:arma}.
                     \end{gathered}
                 \end{equation}
                 Vermeersch and De Moor~\cite{vermeersch2019globally} showed that the globally optimal solution to the ARMA$(p,1)$ identification problem is given by an eigenvalue of a quadratic
                 $(p+1)$-parameter rMEP. The corresponding rMEPs for
                 ARMA(1,1) and ARMA(2,1) are, respectively,
                 \[
                     \left(\rmat_{00} + \alpha_1 \rmat_{10} + \gamma \rmat_{01} + \gamma^2 \rmat_{02}\right) \rvec = \vc{0}
                 \]
                 with structured $(3N - 1) \times (3N - 2)$ coefficient matrices, and
                 \[
                     \left(\rmat_{000} + \alpha_1 \rmat_{100} + \alpha_2 \rmat_{010} + \gamma \rmat_{001} + \gamma^2 \rmat_{002}\right) \rvec = \vc{0}
                 \]
                 with $(4N - 5) \times (4N - 7)$ coefficient matrices, for the
                 exact construction of which see~\cite{vermeersch2019globally}.

                The first numerical method to find
                the optimal solution was based on block Macaulay matrices \cite{VermeerschDeMoorLAA}, while a more efficient approach, which exploits the structure of the support to compress the rMEP before solving~\cite{HKP_NLAA_2023_RectMEP}, is provided as functions
                {\tt arma11} and {\tt arma21} in MultiParEig \cite{multipareig_29}.

                 For ARMA(1,1) and ARMA(2,1) we
                 compared
                 the homotopy method to {\tt macaulaylab} from \cite{macaulaylab} and
                \verb|arma11| and \verb|arma21| from \cite{multipareig_29} on random data sets, see Fig.~\ref{fig:arma}.
                 For ARMA(1,1), the method \verb|arma11| was the fastest for all tested data lengths, but due to huge memory requirements, it would eventually fail for $N$ too large. Although being too slow for real-world problems, the homotopy method can be applied to much larger samples. For ARMA(2,1), the homotopy becomes faster than \verb|arma21| at $N=15$, which is also the largest size that \verb|arma21| could handle without running out of memory.

                \begin{figure}[h]
                    \centering
                    \includegraphics[width=6.4cm]{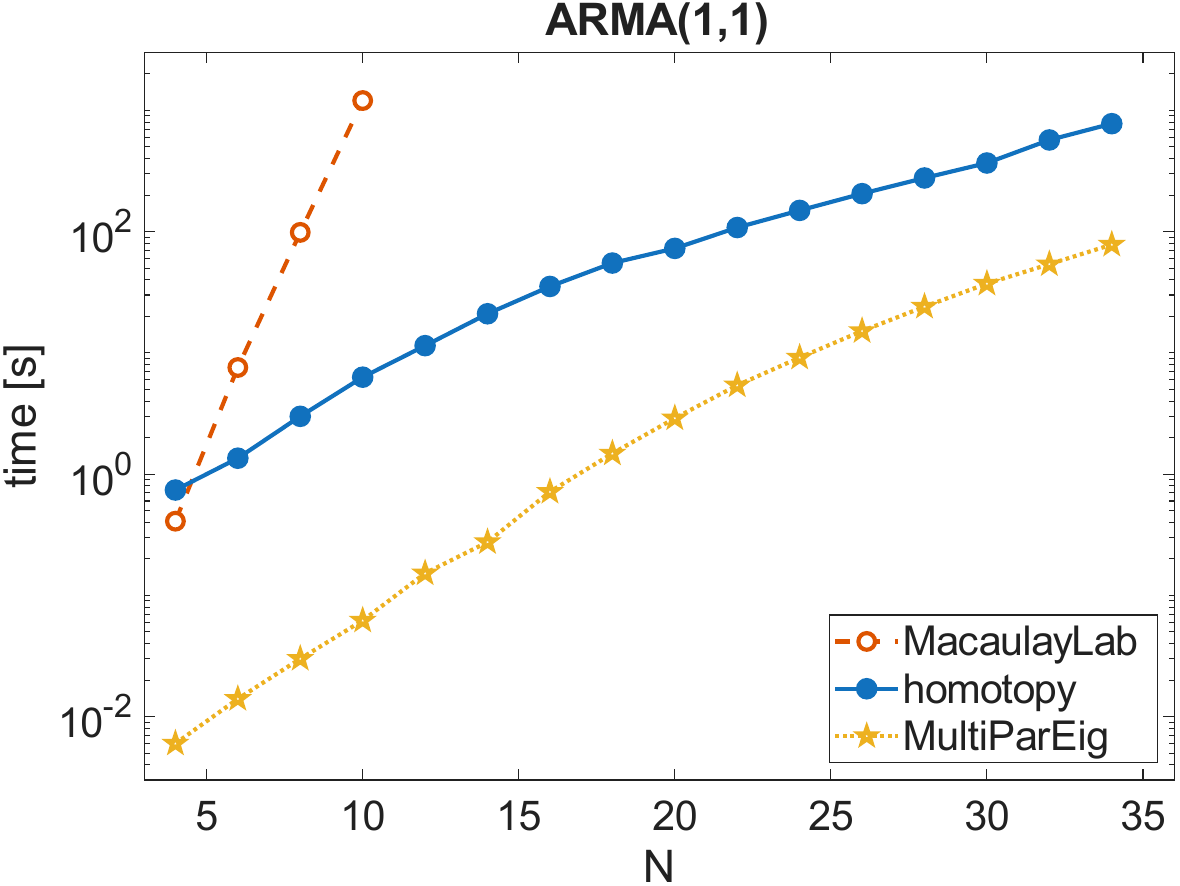}\ \ \
                    \includegraphics[width=6.4cm]{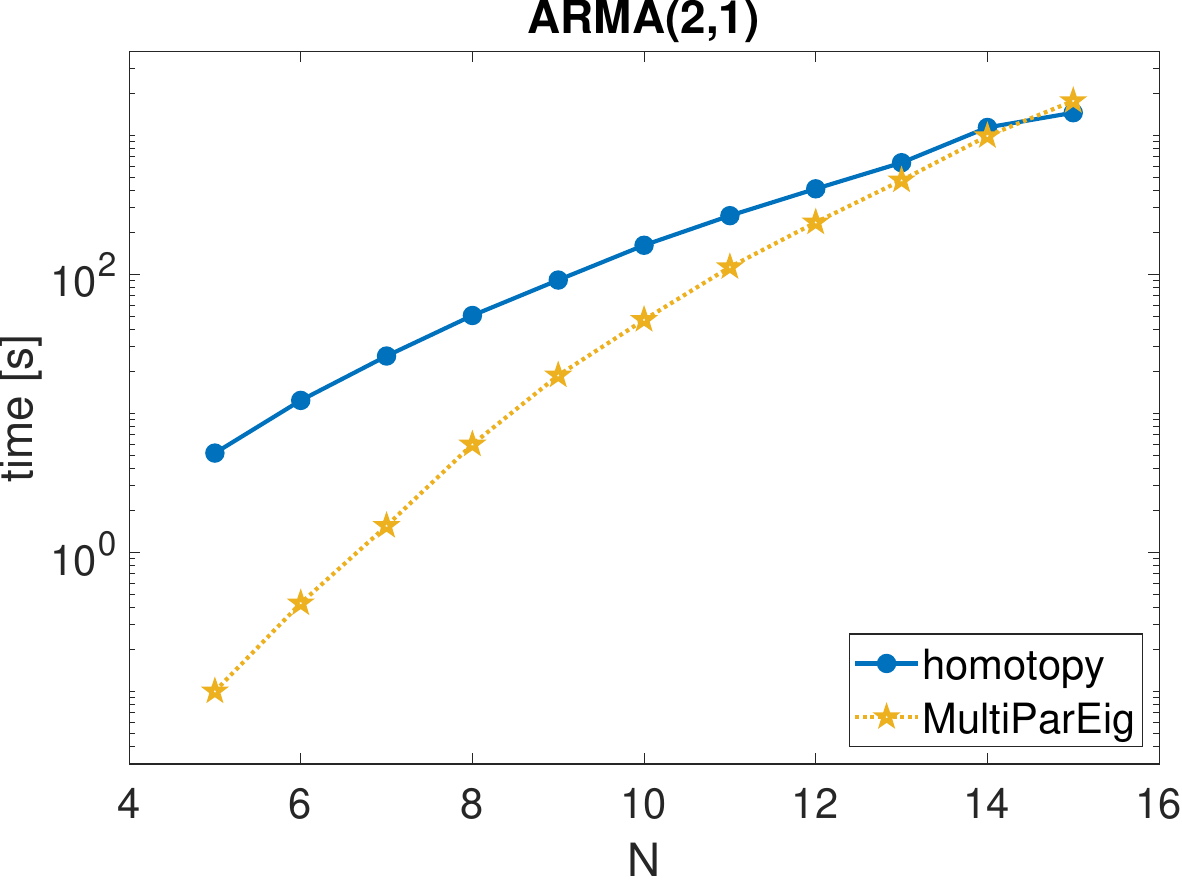}
                    \caption{Timings for finding the optimal solution of ARMA(1,1) and ARMA(2,1) problems
                    using the homotopy method, \texttt{macaulaylab}, \texttt{arma11}, and
                    \texttt{arma21}.\label{fig:arma}}
                \end{figure}
                            \end{example}

            As observed in \cite{HKP_NLAA_2023_RectMEP}, the rMEPs related to ARMA models
            have just a small portion of finite eigenvalues. In the homotopy most paths converge  to points on a
            positive-dimensional set at infinity. If one could construct a homotopy where we
            could follow only paths leading to finite eigenvalues, this would present a major speedup.

            \begin{example}
                Finally, we consider the model order reduction problem for the class of minimal, stable, single-input single-output, linear time-invariant continuous-time systems.
                Given a strictly proper $n$th-order transfer function,
                \begin{equation}
                    H(s) = \frac{b(s)}{a(s)} = \frac{b_{n-1} s^{n-1} + \cdots + b_1 s + b_0}{s^n + a_{n-1} s^{n-1} + \cdots + a_1 s + a_0},
                \end{equation}
                the model order reduction problem consists in searching for a stable reduced model
                $\widehat{H}(s)$
                of order
                $m \ll n$ that minimizes
                \[
                    \label{eq:approximationerror}
                    \norm{H(s) - \widehat{H}(s)}_{\mathcal{H}_2}^2 = \frac{1}{2 \pi} \int_{-\infty}^{+\infty} \left\vert H(j\omega) - \widehat{H}(j\omega)\right\vert^2 \mathrm{d}\omega.
                \]
                Several techniques exist for obtaining such $\widehat{H}(s)$, and recently Lagauw et al.~\cite{lagauw2023globally} showed that a quadratic $m$-parameter rMEP with $(m + n) \times (n + 1)$ coefficient matrices yields the globally optimal model parameters as one of its eigenvalues.

                We run the homotopy method on the $m$-parameter rMEPs arising from the model order reduction problem for $m=2,3,4$ and various $n$, and as before, compare with \texttt{rect\_quad\_multipareig} and {\tt macaulaylab}. Computation times for selected combinations of $m$ and $n$ together with the number of paths for the homotopy and the sizes of the $\Delta$ matrices related to \texttt{rect\_quad\_multipareig} are presented in
                Table \ref{tab:ex5}. For all three values of $m$ homotopy is
                the fastest method for sufficiently large $n$.
                As in the ARMA problems just a small portion
                of paths lead to finite eigenvalues.

\begin{table}[htb!]
\centering
\caption{Computation times in seconds of
    the homotopy method, MacaulayLab, and MultiParEig to solve a quadratic rMEP related to optimal $\mathcal{H}_2$ model reduction
    of a SISO LTI system from order $n$ to order $m$.\label{tab:ex5}}
 \footnotesize
\begin{tabular}{rr|r|rr|rr} \hline \rule{0pt}{2.6ex}%
$m$ & $n$ & MacaulayLab & MultiParEig & size of $\Delta$ & Homotopy & number of paths\\[0.5mm]
\hline \rule{0pt}{2.6ex}%
  2 & 10 & $0.60$ & $\mathbf{0.09}$ & 396 & $2.04$ & $264$ \\
  2 & 30 & $1329.8$ & $\mathbf{9.4}$ & $2976$ & $35.2$ & $1984$ \\
  2 & 50 & out of memory & $\mathbf{192.4}$ & $7956$ & $229.7$ & $5304$ \\
  2 & 70 & out of memory & $1681.7$ & $15336$ & $\mathbf{856.1}$ & $10224$\\
  3 & 5  & $\mathbf{0.06}$ & $0.48$ & $1120$ & $3.23$ & $448$\\
  3 & 10 & $34.7$ & $52.3$ & $5720$ & $\mathbf{20.0}$ & $2288$\\
  3 & 15 &  $1633.3$ & $988.7$ & $16320$ & $\mathbf{85.2}$ & $6528$ \\
  4 & 5  &  $\mathbf{1.5}$ & $1087.8$ & $8820$ & $29.6$ & $2016$ \\
  4 & 10 &  $1357.4$ & out of memory & $70070$ & $\mathbf{181.5}$ & $16016$\\
  4 & 15 &  out of memory & out of memory & $271320$ & $\mathbf{967.8}$ & $62016$\\
\hline
\end{tabular}
\end{table}

                For the final test, we
                applied the homotopy to the model of order $n=17$ from
                \cite[Example 3]{lagauw2023globally}, looking for the best
                reduction of order $m=4$. The computation involving $95760$ paths took $3398$ seconds, compared with $12360$ reported in \cite{lagauw2023globally}.
                In Table \ref{tab:ex_17_4}, we give the poles
                of the five stationary points found with the
                smallest value of $J=\norm{H - \widehat{H}}_{\mathcal{H}_2}$. The list includes a
                point that is not reported in \cite[Table III]{lagauw2023globally}.

\begin{table}[htb!]
\centering
\caption{The poles $p_i$ of the five reduced models of order $4$ with the smallest value
of $J$ for the model from \cite[Example 3]{lagauw2023globally}.\label{tab:ex_17_4}}
 \footnotesize
\begin{tabular}{r|r|r|r} \hline \rule{0pt}{2.6ex}%
$J$ & $p_{1,2}$ & $p_{3,4}$ &
$J$ in \cite{lagauw2023globally}
\\[0.5mm]
\hline \rule{0pt}{2.6ex}%
  $9.0242\cdot 10^{-3}$ & $-0.032009 \pm 78.540 i$ & $-1.1092\cdot 10^{-1} \pm 1.5440\cdot 10^{1}i$ & $9.14\cdot 10^{-3}$\\
  $9.0921\cdot 10^{-3}$ & $-0.032005 \pm 78.540 i$ & $-5.7441\cdot 10^{0\phantom{-}} \pm 5.2614\cdot 10^{1}i$ & $9.22\cdot 10^{-3}$ \\
  $1.0318\cdot 10^{-2}$ & $-0.032012 \pm 78.540 i$ & $-2.3207\cdot 10^{-2} \pm 3.8413\cdot 10^{0}i$ & $1.03\cdot 10^{-2}$ \\
  $1.0410\cdot 10^{-2}$ & $-0.032012 \pm 78.540 i$ & $-3.4262\cdot 10^{-2} \pm 4.9079\cdot 10^{0}i$ & not found\\
  $1.0755\cdot 10^{-2}$ & $-0.032004 \pm 78.540 i$ & $-4.5937\cdot 10^{0\phantom{-}} \pm 1.5917\cdot 10^{1}i$ & $1.09\cdot 10^{-2}$ \\
\hline
\end{tabular}
\end{table}

            \end{example}

    \section{Conclusion and future work}
        \label{sec:conclusion}

        We presented, to the best of our knowledge, the first true homotopy method for rectangular multiparameter eigenvalue problems (rMEPs).
        A key contribution in that effort is the construction of suitable initial problems with the generic number of solutions.
        The accompanying theoretical results provide additional insights into the properties of eigenvalues of rMEPs.

        The novel homotopy method was successfully applied to both linear and polynomial rMEPs with zero-dimensional solution sets.
        Moreover, numerical results indicate that the
        method can be applied to problems with positive-dimensional solution sets, although a rigorous theoretical justification for the positive-dimensional case is still missing.
        This case is particularly relevant for applications in system identification and model order reduction.

        Possible directions for future work include the construction of alternative initial problems for non-generic structures and efficient computation of their solutions, and the development of a supporting theory for problems with positive-dimensional solution sets.

\medskip

\noindent\textbf{Author Contributions}\quad All authors contributed equally to the paper.
\smallskip

\noindent\textbf{Funding}\quad {Funding was provided by Javna agencija za znanstvenoraziskovalno in inovacijsko dejavnost Republike Slovenije (Grant No. P1-0294).}
\smallskip

\noindent\textbf{Code availability}\quad The code and data for numerical examples in this paper are available
at \url{https://github.com/borplestenjak/RMEP_homotopy}.
\smallskip

\section*{Declarations}

\noindent\textbf{Conflict of interests}\quad Not applicable.




    \appendix

\section{Regularity of the homotopy}\label{sec:homotopy_details}

In a similar way as in \cite{DongYuYu_SIMAX} we can show that for a finitely
regular linear rMEP the homotopy paths do not intersect on $t\in[0,1)$.

\begin{lemma}\label{lem:fullrank}
    Let $\rmep{\eig} = \big(\rmat_0 + \sum_{j=1}^k \lambda_j \rmat_j\big)$
    and $\irmep{\eig} = \big(\irmat_0 +  \sum_{j=1}^k \lambda_j\irmat_j\big)$ be finitely regular rMEPs with $(n+k-1)\times n$ matrices.
    Then there exists an open dense full-measure set $\mathcal{U}_1\subset \CC$ such that for
    all $\gamma\in\mathcal{U}_1$
    the rMEP
    \[\mcl{H}(t,\eig) = \bigg(\mt{H}_0(t) + \sum_{j = 1}^k \lambda_j \mt{H}_j(t)\bigg),\]
    where $\mt{H}_j(t) = (1 - t) \gamma \irmat_j + t \rmat_j$ for $j=0,\ldots,k$,
    is finitely regular
    for all $t\in[0,1]$.
\end{lemma}
\begin{proof}
Let $h_{j_1, \ldots, j_n}(t,\eig) \coloneq \det(\mt{H}_{j_1, \ldots, j_n}(t,\eig))$ be an
            $n \times n$ minor of $\sum_{j = 1}^k \lambda_j \mt{H}_j(t)$ composed from the rows
            $1 \le j_1 < \ldots <j_n \le n + k - 1$.
            Then
            \[
                \rank(\lambda_1 \mt{H}_1(t) + \cdots + \lambda_k \mt{H}_k(t))<n
            \]
            if and only if
            $
                \label{eq:eig_system_minors_app}
                h_{j_1, \ldots, j_n}(t,\eig) = 0
            $
            for all $1 \le j_1 < \ldots <j_n \le n + k - 1$.
Since $\mcl{H}(1,\eig)=\rmep{\eig}$ is finitely regular,
we can assume $0\le t<1$,
introduce \begin{equation}\label{eq:s_gamma}
s=\frac{t}{(1-t)\gamma},
\end{equation}
which is well defined for $t\ne 1$ and $\gamma\ne 0$, and write
$\mt{H}_j(t) = (1 - t) \gamma (\irmat_j + s \rmat_j)$. It follows
that $h_{j_1, \ldots, j_n}(t,\eig)=(1-t)^n\gamma^n g_{j_1, \ldots, j_n}(s,\eig)$, where
$g_{j_1, \ldots, j_n}$ is a polynomial in $\lambda_1,\ldots,\lambda_k$ and $s$.
Since we consider $\eig\ne\vc{0}$, we can treat $\eig$ as an element of
the projective space $\CP^{k-1}$ and define an
algebraic set
\[\mathcal{V}_1=\{ (s,\eig)\in \CC\times \CP^{k-1}:\ g_{j_1, \ldots, j_n}(s,\eig)=0\ \textrm{for all }
1 \le j_1 < \ldots <j_n \le n + k - 1\}.
\]
The projection of $\mathcal{V}_1$ onto $\CC$,
\[
\mathcal{S}_1=\{ s\in \CC:\ \textrm{there exists }\eig\in\CP^{k-1}\ \textrm{such that } (s,\eig)\in\mathcal{V}_1\},\]
is an algebraic subset of $\CC$.
Since $\irmep{\eig}$
is finitely regular,  $0\not\in \mathcal{S}_1$ and
$\mathcal{S}_1$ is finite. For a fixed $\gamma\ne 0$ and $0\le t<1$, points \eqref{eq:s_gamma} form a ray and, if we take $\gamma\in\mathcal{U}_1$, where
\begin{equation}\label{alg_set_U}
\mathcal{U}_1=\{\gamma\in\CC:\ \gamma\ne 0\ \textrm{and }\arg(\gamma)\ne -\arg(s)\ \textrm{for all }s\in\mathcal{S}_1
\}\end{equation}
is an open dense full-measure subset of $\CC$,
then all points on the ray miss $\mathcal{S}_1$.
\end{proof}

\begin{lemma}\label{lem:simpleeigs}
Let rMEPs $\rmep{\eig}$ and $\irmep{\eig}$ satisfy the assumptions of
    Lemma \ref{lem:fullrank} and let all
    eigenvalues of $\irmep{\eig}$ be simple.
    Then there exists an open dense full-measure set $\mathcal{U}_2\subset \CC$ such that for
    all $\gamma\in\mathcal{U}_2$ all
    eigenvalues of the
    rMEP $\mcl{H}(t,\eig)$
    are simple for all $t\in[0,1)$.
\end{lemma}

\begin{proof}
We use similar tools as in the previous proof. We introduce
parameter $s$ as in \eqref{eq:s_gamma} and consider $\mcl{H}(t,\eig)$ in
the homogeneous form
\[ {\mcl{H}}(s,\heig)\rvec  = (1-t)\gamma\left(\sum_{j = 0}^k \lambda_j (\irmat_j+s \rmat_j)\right)\rvec=\vc{0}\]
by introducing parameter $\lambda_0$.

Let $\heig$ be an eigenvalue of ${\mcl{H}}(s,\heig)$. Then
$g_{j_1, \ldots, j_n}(s,\heig)=0$ for all
$1 \le j_1 < \ldots <j_n \le n + k - 1$, where
$g_{j_1, \ldots, j_n}$ is a polynomial in $\lambda_0,\ldots,\lambda_k$ and $s$,
defined similarly to those in the previous proof. In addition, if $\heig$ is a multiple eigenvalue, then the rank of the corresponding Jacobian
$J_g(s,\heig)$,
which is an $\binom{n+k-1}{n} \times (k+1)$ matrix with rows
            \[
                \left[\begin{matrix}
                    \frac{\partial g_{j_1,\ldots,j_n}}{\partial \lambda_0}(\heig) & \cdots & \frac{\partial g_{j_1,\ldots,j_n}}{\partial \lambda_k}(\heig)
                \end{matrix}\right],\quad 1\le j_1<\ldots<j_n\le n+k-1,
            \]
is less than $k$.
Now we define $\mathcal{V}_2\subset \CC\times \CP^{k}$ as the set of all points
$(s,\heig)$ such that $g_{j_1, \ldots, j_n}(s,\heig)=0$ for all
$1 \le j_1 < \ldots <j_n \le n + k - 1$ and $\rank(J_g(s,\heig))<k$.
As these are algebraic constraints, $\mathcal{V}_2$ is algebraic and so is
its projection to $\CC$, defined as
\[\mathcal{S}_2=\{ s\in \CC:\ \textrm{there exists }\heig\in\CP^{k}\ \textrm{such that } (s,\heig)\in\mathcal{V}_2\}.
\]
Since $\irmep{\eig}$ is finitely regular, all
eigenvalues of ${\mcl{H}}(0,\heig)$ are such that $\lambda_0\ne 0$. Since
we assume that all eigenvalues of $\irmep{\eig}$ are simple,
there does not exist $\heig\in\CP^{k}$ such that $(0,\heig)\in\mathcal{V}_2$.
Thus, as before, $0\not\in \mathcal{S}_2$, so $\mathcal{S}_2$ is finite and in a similar way as in
\eqref{alg_set_U} we obtain $\mathcal{U}_2$.
\end{proof}

Suppose that $\rmep{\eig}$ is a finitely regular linear rMEP.
Since singular points of the homotopy from Section \ref{sec:linear} correspond to multiple eigenvalues, it follows that
there exists an open dense full-measure set $\mathcal{U}\subset \CC$ such that for
$\gamma\in\mathcal{U}$ the homotopy paths on $t\in[0,1)$ do not cross and do not contain
any singular points.

\section{Selection of the initial rMEP}\label{sec:optchoice}

We want to select random nonzero entries of the initial linear rMEP in such a way that the
condition numbers of the starting eigenvalues are as small as possible.

We consider the construction \eqref{eq:twopolynomialmatrix} for $k=2$.
If we denote $b_i^{(0)}=\alpha_i$ and $b_i^{(1)}=\beta_i$ for $i=1,\ldots,n$, then
$\eig^{(j\ell)}=(-\alpha_j,-\beta_\ell)$ is an eigenvalue for
 $1\le j\le \ell \le n$.
The corresponding right
eigenvector $\vc{x}^{(j\ell)}$ of size $n$ has the form
\[ \vc{x}^{(j\ell)} = \begin{bmatrix}0 & \cdots & 0 & x_j & \cdots & x_\ell & 0 & \cdots & 0\end{bmatrix}^T,
\]
while orthogonal vectors of size $n+1$ of the form
\[ \vc{y}^{(j\ell)} = \begin{bmatrix}y_1 & \cdots & y_j & 0 & \cdots & 0\end{bmatrix}^T\
\textrm{and}\ \
 \vc{z}^{(j\ell)} = \begin{bmatrix}0 & \cdots & 0 & z_{\ell+1} & \cdots & z_{n+1}\end{bmatrix}^T
\]
span the two-dimensional subspace $\ker(\irmep{\eig^{(j\ell)}}^\herm)$.

The condition number of
$\eig^{(j\ell)}$ is related to
the condition number of the matrix \eqref{genB0_conditionB} from Corollary \ref{thm:simpleeig}, which is
for the eigenvalue $\eig^{(j\ell)}$
 equal to
\[ \begin{bmatrix} {\vc{y}^{(j\ell)}}^*B_1  \vc{x}^{(j\ell)} & {\vc{y}^{(j\ell)}}^*B_2  \vc{x}^{(j\ell)} \cr
                            {\vc{z}^{(j\ell)}}^*B_1  \vc{x}^{(j\ell)} & {\vc{z}^{(j\ell)}}^*B_2  \vc{x}^{(j\ell)}\end{bmatrix} =
                            \begin{bmatrix} \overline{y_j}x_j & 0 \cr 0 & \overline{z_{\ell+1}}x_\ell\end{bmatrix},
\]
and is best conditioned when its diagonal entries have equal absolute value,
i.e., when
\begin{equation}\label{eq:condest}
   \left|\frac{y_jx_j}{z_{\ell+1}x_\ell}\right|\approx 1.
\end{equation}

If we scale $\vc{y}^{(j\ell)}$ so that $y_1=1$, then
\[|y_1|=1,\ |y_2|=\left|\frac{\alpha_1-\alpha_j}{\beta_1-\beta_{\ell}}\right|,\ldots,
|y_j|=\left|\frac{(\alpha_1-\alpha_j)\cdots(\alpha_{j-1}-\alpha_j)}{(\beta_1-\beta_{\ell})\cdots
(\beta_{j-1}-\beta_\ell)}\right|.
\]
In a similar way, if we scale $\vc{z}^{(j\ell)}$ so that $z_{n+1}=1$, then
\[ |z_{\ell+1}|=\left|\frac{(\beta_{\ell+1}-\beta_\ell)\cdots(\beta_{n}-\beta_\ell)}
{(\alpha_{\ell+1}-\alpha_j)\cdots(\alpha_{n}-\alpha_j)}\right|,\ldots,
|z_n|=\left|\frac{\beta_n-\beta_\ell}{\alpha_n-\alpha_j}\right|,
|z_{n+1}|=1.
\]
Finally, if we scale $\vc{x}^{(j\ell)}$ so that $x_j=1$, then
\[
|x_\ell|=\left|\frac{(\beta_j-\beta_{\ell})\cdots
(\beta_{\ell-1}-\beta_\ell)}{(\alpha_{j+1}-\alpha_j)\cdots(\alpha_{\ell}-\alpha_j)}\right|.
\]
The expression \eqref{eq:condest}, after we normalize all vectors, is equal to
\begin{equation}\label{eq:condjl}
\frac{|y_1^{(j\ell)}|}{|z_{n+1}^{(j\ell)}|}
\left|
\frac{\prod_{i=1,\ i\ne j}^n(\alpha_i-\alpha_j)}
{\prod_{i=1,\ i\ne \ell}^n(\beta_i-\beta_\ell)}
\right|.
\end{equation}
We see that \eqref{eq:condjl} might be very small or large if there exist $i\ne j$
such that $|\alpha_i-\alpha_j|\ll 1$ or $k\ne \ell$ such that $|\beta_k-\beta_\ell|\ll 1$.
The heuristic approach based on the above is
to select $\alpha_1,\ldots,\alpha_n$ and $\beta_1,\ldots,\beta_n$ as a random permutation of the $n$-th roots of unity, multiplied by a random unimodular complex scalar.

A straightforward generalization of this selection performs well for rMEPs with more than two parameters.

\end{document}